\documentclass[a4paper,12pt]{amsart} %
\usepackage[utf8x]{inputenc}
\usepackage{amsmath}
\usepackage{amsfonts}
\usepackage{amssymb}
\usepackage[all]{xy}
\usepackage{verbatim}

\theoremstyle{plain}
\newtheorem{theorem}{Theorem}[section]
\newtheorem{lemma}[theorem]{Lemma}
\newtheorem{corollary}[theorem]{Corollary}
\newtheorem{proposition}[theorem]{Proposition}

\newtheorem{example}[theorem]{Example}

\def\GL{\operatorname{GL}}

\def\d{\operatorname{d\!}}

\def\Diff{\operatorname{Diff}}

\def\exp{\operatorname{exp}}

\def\Gr{\operatorname{Gr}}

\def\Hom{\operatorname{Hom}}

\def\im{\operatorname{im}}
\def\Im{\operatorname{Im}}
\def\Ind{\operatorname{Ind}}

\def\ker{\operatorname{ker}}

\def\max{\operatorname{max}}
\def\mod{\operatorname{mod}}

\def\res{\operatorname{res}}

\def\span{\operatorname{span}}

\def\Stab{\operatorname{Stab}}
\def\supp{\operatorname{supp}}

\begin{document}

\title[Casselman's comparison theorem]{A proof of Casselman's comparison theorem}
\author{Ning Li}
\address{Ning Li, Beijing International Center for Mathematical Research, Peking University, No. 5 Yiheyuan Road,
Beijing 100871, China}
\email{lining@bicmr.pku.edu.cn}
\author{Gang Liu}
\address{Gang Liu, Institut Elie Cartan de Lorraine, CNRS-UMR 7502, Universit\'e de Lorraine,  3 rue Augustin Fresnel,
57045 Metz, France}
\email{gang.liu@univ-lorraine.fr}
\author{Jun Yu}
\address{Jun Yu, Beijing International Center for Mathematical Research, Peking University, No. 5 Yiheyuan Road,
Beijing 100871, China}
\email{junyu@bicmr.pku.edu.cn}
\begin{abstract}Let $G$ be a real linear reductive group and $K$ be a maximal compact subgroup. Let $P$ be a minimal
parabolic subgroup of $G$ with complexified Lie algebra $\mathfrak{p}$, and $\mathfrak{n}$ be its nilradical. In this
paper we show that: for any admissible finitely generated moderate growth smooth Fr\'echet representation $V$ of $G$,
the inclusion $V_{K}\subset V$ induces isomorphisms $H_{i}(\mathfrak{n},V_{K})\cong H_{i}(\mathfrak{n},V)$ ($i\geq 0$),
where $V_{K}$ denotes the $(\mathfrak{g},K)$ module of $K$ finite vectors in $V$. This is called Casselman's comparison
theorem (\cite{Hecht-Taylor3}). As a consequence, we show that: for any $k\geq 1$, $\mathfrak{n}^{k}V$ is a closed
subspace of $V$ and the inclusion $V_{K}\subset V$ induces an isomorphism $V_{K}/\mathfrak{n}^{k}V_{K}=
V/\mathfrak{n}^{k}V$. This strengthens Casselman's automatic continuity theorem (\cite{Casselman},\cite{Wallach}).
\end{abstract}

\maketitle

\setcounter{tocdepth}{1}

\tableofcontents

\section*{Introduction}

Let $G$ be a real linear reductive group and $K$ be a maximal compact subgroup. Let $W$ be a finitely
generated admissible $(\mathfrak{g},K)$-module and $W^{\ast}$ be its dual $(\mathfrak{g},K)$-module. Schmid constructed
a minimal globalization $W_{\min}$ which consists of analytic vectors, and the topological dual of $(W^{\ast})_{\min}$
is called a maximal globalization of $W$ (\cite{Schmid}). Casselman and Wallach constructed a smooth globalization
$W_{\infty}$ (also called a canonical globalization or a Casselman-Wallach globalization, \cite{Casselman},
\cite{Wallach}) which is a moderate growth smooth Fr\'echet representation of $G$, and one has a globalization
$W_{-\infty}$ as the topological dual of $(W^{\ast})_{\infty}$. They fit into a sequence \[W\subset W_{\min}
\subset W_{\infty}\subset W_{-\infty}\subset W_{\max}.\] Each kind of globalization has important applications. For
example, the smooth globalization $W_{\infty}$ is very useful in the theory of automorphic forms, and the maximal
globalization $W_{\max}$ is used to realize standard derived functor modules (\cite{Schmid-Wolf}).

Let's recall the Casselman-Wallach globalization (\cite{Casselman}, \cite{Wallach}, \cite{Bernstein-Krotz}). The
presentation in this paragraph mainly follows \cite{Wallach2}. A $(\mathfrak{g},K)$-module $W$ is said to be admissible if
$\dim\Hom_{K}(\tau,W)<\infty$ for any irreducible complex linear representation $\tau$ of $K$; it is said to be finitely
generated if there exists a finite dimensional subspace $F\subset W$ such that $W=\mathcal{U}(\mathfrak{g})F$. For a
continuous representation $(\pi,V)$ of $G$ on a topological vector space $V$, let $V_{K}$ be the subspace of $K$-finite
vectors, i.e., \[V_{K}=\{v\in V:\dim\span_{\mathbb{C}}\{\pi(k)v:k\in K\}<\infty\}.\] If $\dim\Hom_{K}(\tau,V)<\infty$ for
any irreducible complex linear representation $\tau$ of $K$, then we call $V$ admissible. In this case $V_{K}$ consists
of smooth vectors and it is an admissible $(\mathfrak{g},K)$-module (\cite[Proposition 8.5]{Knapp}). An admissible
continuous representation $(\pi,V)$ is said to be finitely generated if $V_{K}$ is finitely generated; it is said to be
a smooth Fr\'echet representation if $V$ is a Fr\'echet space and for any $v\in V$, the map $g\mapsto\pi(g)v$ is of class
$C^{\infty}$ on $G$. For a complex square matrix $X=\{a_{i,j}\}_{1\leq i,j\leq n}$, the Hilbert-Schmidt norm of $X$ is
defined by \[||X||=\sqrt{\sum_{1\leq i,j\leq n}|a_{i,j}|^{2}}.\] Fix a linear group imbedding $p:G\rightarrow
\GL(n,\mathbb{C})$. For any $g\in G$, we define a norm $p$ on $G$ by \[||g||=\max\{||p(g)||,||p(g)^{-1}||\},\] which
satisfies the following properties:
\begin{itemize}
\item[(i)]$||g||\geq 1$ ($\forall g\in G$);
\item[(ii)]$||xy||\leq||x||\ ||y||$ ($\forall x,y\in G$);
\item[(iii)]For any $t\in\mathbb{R}$, $\{g\in G:||g||\leq t\}$ is a compact set.
\end{itemize}
A smooth Fr\'echet representation $(\pi,V)$ is said to be moderate growth if for every continuous seminorm $\lambda$ on
$V$, there exists a continuous seminorm $\mu$ on $V$ and a constant $r\in\mathbb{R}$ such that \[\lambda(\pi(g)v)\leq
||g||^{r}\mu(v),\ \forall g\in G, \forall v\in V.\] Write $\mathcal{H}(\mathfrak{g},K)$ for the category of
$(\mathfrak{g},K)$-modules that are admissible and finitely generated; write $\mathcal{HF}_{\mod}(G)$ for the category
of admissible finitely generated moderate growth smooth Fr\'echet representations of $G$. The celebrated
Casselman-Wallach theorem asserts that the functor \[\mathcal{HF}_{\mod}(G)\longrightarrow\mathcal{H}(\mathfrak{g},K),
\quad V\mapsto V_{K}\] is an equivalence of categories.

Let $P$ be a minimal parabolic subgroup of $G$ with complexfied Lie algebra $\mathfrak{p}$. Let $\mathfrak{n}$ be
the nilradical of $\mathfrak{p}$. In this paper we show that: for any admissible finitely generated moderate growth smooth
Fr\'echet representation $V$ of $G$, the inclusion $V_{K}\subset V$ induces isomorphisms $H_{i}(\mathfrak{n},V_{K})
\cong H_{i}(\mathfrak{n},V)$ ($i\geq 0$). This is called Casselman's comparison theorem (\cite{Hecht-Taylor3}). As a
consequence, we show that: for any $k\geq 1$, $\mathfrak{n}^{k}V$ is a closed subspace of $V$ and the inclusion
$V_{K}\subset V$ induces an isomorphism $V_{K}/\mathfrak{n}^{k}V_{K}=V/\mathfrak{n}^{k}V$. This implies Casselman's
automatic continuity theorem saying that $V_{K}/\mathfrak{n}^{k}V_{K}=V/\overline{\mathfrak{n}^{k}V}$ for each $k\geq 0$
(\cite[p. 416]{Casselman}). Let $V'$ be the strong dual of $V$ (\cite{Schaefer}), which is a dual nuclear Fr\'echet space.
Our proof uses the {\it Casselman-Jacquet module} \[V^{'[\mathfrak{n}]}=\{u\in V':\exists k\in
\mathbb{Z}_{>0},\mathfrak{n'}^{k}\cdot u=0\}\] of $V'$ in an essential way.

Let's describe the strategy of our proof briefly. It is well-known that $H_{i}(\mathfrak{n},V_{K})$ and $H^{i}
(\mathfrak{n},V_{K}^{\ast})$ are finite-dimensional (\cite[Corollary 2.4]{Casselman-Osborne}). By a duality argument
(\cite[Lemma 5.11]{Casselman-Hecht-Milicic}), the homological comparison theorem $H_{i}(\mathfrak{n},V_{K})
\cong H_{i}(\mathfrak{n},V)$ ($i\geq 0$) and the cohomological comparison theorem $H^{i}(\mathfrak{n},V_{K}^{\ast})
\cong H^{i}(\mathfrak{n},V')$ ($i\geq 0$) are equivalent, where $V_{K}^{\ast}$ denotes the dual space of $V_{K}$.
Moreover, a reduction argument in \cite[Proposition 3]{Hecht-Taylor3} reduces the homological comparison theorem
to the principal series case. Let $V=I(\sigma):=\Ind_{P}^{G}(\sigma)$ (un-normalized smooth parabolic induction)
be a principal series, where $(\sigma,V_{\sigma})$ is a finite-dimensional complex linear algebraic representation
of $P$. By a theorem of Hecht-Schmid (\cite[Lemma 2.37]{Hecht-Schmid}), we have $H^{i}(\mathfrak{n},V_{K}^{\ast})=
H^{i}(\mathfrak{n},(V_{K}^{\ast})^{[\mathfrak{n}]})$. By Casselman's automatic continuity theorem
(\cite[p. 416]{Casselman},\cite[Theorem 11.4]{Bernstein-Krotz}, \cite[p.77]{Wallach}), we have $V^{'[\mathfrak{n}]}
=(V_{K}^{\ast})^{[\mathfrak{n}]}$. Then, it reduces to the following assertion: the inclusion
$I(\sigma)^{'[\mathfrak{n}]}\subset I(\sigma)'$ induces isomorphisms \begin{equation}\label{Eq:comparison1}
H^{i}(\mathfrak{n},I(\sigma)^{'[\mathfrak{n}]})=H^{i}(\mathfrak{n},I(\sigma)'),\ \forall i\in\mathbb{Z}.
\end{equation}

Choose a maximal split torus $A$ contained in $P$. Let $W=N_{G}(A)/Z_{G}(A)$ be the Weyl group of $G$. For each element
$w\in W$, choose an element $\dot{w}\in K$ representing $w$ and put $C(w)=N\dot{w}P\subset G/P$, which is called the
Bruhat cell attached to $w$. Write $W=\{w_{1},\dots,w_{r}\}$ so that $\dim C(w_{i})\leq\dim C(w_{i+1})$
($1\leq i\leq r-1$). For each $k$ ($0\leq k\leq r$), write \[Z_{k}=\bigcup_{1\leq i\leq k}C(w_{i}).\] Then, by the
Bruhat decomposition we have a stratification \[\emptyset=Z_{0}\subset Z_{1}\subset\cdots\subset Z_{r}=X\] of the flag
variety $X:=G/P$. Put \[E(\sigma)=G\times_{P}V_{\sigma}\] for a $G$ equivariant bundle on $X$ induced from the $P$
representation $V_{\sigma}$. Then, $C^{\infty}(E(\sigma))=I(\sigma)$ and $C^{\infty}(E(\sigma))'=I(\sigma)'$. For each
$k$, let $I_{k,\sigma}$ be the space of $E(\sigma)$-distributions with support contained in $Z_{k}$. For each $w\in W$,
in \cite{Casselman-Hecht-Milicic} the authors defined a space
$\mathcal{T}(C(w),E(\sigma))$ of {\it tempered $E(\sigma)$-distributions on an open subanalytic neighborhood of
$C(w)$} with support in $C(w)$. Then, \[I_{k,\sigma}/I_{k-1,\sigma}=\mathcal{T}(C(w_{k}),E(\sigma))\] for each $k$.
For each integer $p\geq 0$, in \cite{Casselman-Hecht-Milicic} the authors defined a space $F_{p}\mathcal{T}(C(w),
E(\sigma))$ of distributions in $\mathcal{T}(C(w),E(\sigma))$ of {\it transversal degree $\leq p$}. Put
\[\Gr^{p}\mathcal{T}(C(w),E(\sigma))=F_{p}\mathcal{T}(C(w),E(\sigma))/F_{p-1}\mathcal{T}(C(w),E(\sigma)).\] We show
that \[\mathcal{T}(C(w_{k}),E(\sigma))^{[\mathfrak{n}]}=I_{k,\sigma}^{[\mathfrak{n}]}/I_{k-1,\sigma}^{[\mathfrak{n}]},
\ \forall k\geq 1\] and \[\Gr^{p}\mathcal{T}(C(w),E(\sigma))^{[\mathfrak{n}]}=F_{p}\mathcal{T}(C(w),
E(\sigma))^{[\mathfrak{n}]}/F_{p-1}\mathcal{T}(C(w),E(\sigma))^{[\mathfrak{n}]},\ \forall p\geq 0.\] When $N$ acts
trivially on $V_{\sigma}$, we show that each $\Gr^{p}\mathcal{T}(C(w),E(\sigma))$ admits a finite increasing filtration
such that each graded piece is isomorphic to $\mathcal{S}(C(w))'$ (the space of Schwarz distributions on $C(w)$ as a
real Euclidean space) and $\mathcal{S}(C(w))^{'[\mathfrak{n}]}=R(C(w))\delta_{C(w)}$, where $R(C(w))$ denotes the space
of regular functions on $C(w)$ (as an affine algebraic variety) and $\delta_{C(w)}$ is an $N$ invariant volume measure
on $C(w)$ (unique up to scalar). Taking induction on $k$ and $p$, \eqref{Eq:comparison1} is reduced to the following
assertion: for each $w\in W$, the inclusion $R(C(w))\delta_{C(w)}\subset\mathcal{S}(C(w))'$ induces isomorphisms
\begin{equation}\label{Eq:comparison2}H^{i}(\mathfrak{n},R(C(w))\delta_{C(w)})=H^{i}(\mathfrak{n},\mathcal{S}(C(w))'),
\ \forall i\in\mathbb{Z}.\end{equation} We prove a more general statement than this (Proposition \ref{P:comparison1})
by a method inspired by ideas in \cite[\S 5]{Casselman-Hecht-Milicic}. For a general finite-dimensional algebraic
representation $(\sigma,V_{\sigma})$ of $P$, we prove \eqref{Eq:comparison1} by induction on the length of the
composition series of $V_{\sigma}$.

For each of the four kinds of globalization above, there is a comparison conjecture for $\mathfrak{n}$-homology
(or $\mathfrak{n}$-cohomology) for $\mathfrak{n}$ the nilradical of certain parabolic subalgebras $\mathfrak{p}$
of $\mathfrak{g}$ (\cite[Conj. 10.3]{Vogan}). Casselman claimed a proof of the comparison theorem when $\mathfrak{p}$
is any standard parabolic subalgebra, but the proof remains unpublished. For minimal globalization, the comparison
conjecture is shown in \cite{Hecht-Taylor}, \cite{Bunke-Olbrich} and \cite{Bratten}. All of these proofs use
sophisticated analytic results about D-modules. In \cite{Hecht-Taylor3} the above comparison theorem for smooth
globalization and for $\mathfrak{p}$ a standard minimal parabolic subalgebra is deduced from the corresponding
comparison theorem for minimal globalization. Compared to these developments, our proof of the comparison theorem
is more elementary and much easier. The closedness of $\mathfrak{n}^{k}V$ and the isomorphism
$V_{K}/\mathfrak{n}^{k}V_{K}=V/\mathfrak{n}^{k}V$ ($\forall k\geq 1$) shown in this paper are new.

\smallskip

\noindent\textbf{Notation and conventions.} Let $G$ be a real Lie group. Write $\mathfrak{g}_{0}$ for the Lie algebra
of $G$, and write $\mathfrak{g}=\mathfrak{g}_{0}\otimes_{\mathbb{R}}\mathbb{C}$ for the complexified Lie algebra of
$G$. Similarly, we have Lie algebras (resp. complexified Lie algebras) $\mathfrak{p}_{0}$, $\mathfrak{n}_{0}$,
$\mathfrak{a}_{0}$ (resp. $\mathfrak{p}$, $\mathfrak{n}$, $\mathfrak{a}$) of Lie groups $P$, $N$, $A$ appearing in this
paper.

For a complex Lie algebra $\mathfrak{h}$, write $\mathcal{U}(\mathfrak{h})$ for the enveloping algebra of $\mathfrak{h}$.
For an integer $p\geq 0$, write $\mathcal{U}_{p}(\mathfrak{h})$ for the subspace of $\mathcal{U}(\mathfrak{h})$
generated by elements $X_{1}\cdots X_{q}\in\mathcal{U}(\mathfrak{h})$ where $X_{1},\dots,X_{q}\in\mathfrak{h}$
and $q\leq p$.

For an affine algebraic variety $Y$ over $\mathbb{R}$, write $R(Y)$ for the ring of complex coefficient regular functions
on $Y$. For each integer $l\geq 0$, write $R_{l}(Y)$ for the space of complex coefficient regular functions on $Y$ of
degree at most $l$.

\smallskip

\noindent\textbf{Acknowledgement.} This work benefits a lot from discussions with Yoshiki Oshima. We are grateful to
Dr. Oshima for inspiring discussions. We would like to thank Professor Bernhard Kr\"otz for useful discussions
and we thank the referee for careful reading and very helpful comments. Jun Yu's research is partially supported by
the NSFC Grant 11971036.

\section{Tempered $E$-distributions supported in a subanalytic submanifold}\label{S:E-distribution}

In this section we review the theory of tempered $E$-distributions supported in a subanalytic submanifold and its transversal
degree filtration defined in \cite[\S 2]{Casselman-Hecht-Milicic}. The exposition in \cite{Casselman-Hecht-Milicic} is
excellent and the paper itself is well known to experts. For most results we only give a sketch of proof. Interested
readers are encouraged to consult the original paper \cite{Casselman-Hecht-Milicic} for more complete account and proof.

Let $X$ be a compact analytic manifold and $E$ be an analytic vector bundle on $X$ of finite rank. Let $C^{\infty}(X,E)$
denote the space of smooth sections of $E$. It is a {\it nuclear Fr\'echet space} (NF space for short) with seminorms
\[p_{D,f}:s\mapsto\max_{x\in X}|f(x)((Ds)(x))|,\] where $f\in C^{\infty}(X,E^{\ast})$ is a smooth section of the
dual vector bundle $E^{\ast}$ of $E$ and $D\in\Diff(X,E)$ is an $E$-valued smooth differential operator on $X$. Let
$C^{\infty}(X,E)'$ be the {\it strong dual} of $C^{\infty}(X,E)$ (\cite{Schaefer}), which is a {\it dual nuclear
Fr\'echet space} (DNF space for short). For an open subset $U$ of $X$, let $C_{0}^{\infty}(U,E)$ be the space of
compactly supported $E$-sections on $U$, which is again a nuclear Fr\'echet space. The strong dual
$C_{0}^{\infty}(U,E)'$ of $C_{0}^{\infty}(U,E)$ is called the space of {\it $E$-distributions on $U$}. There is a
natural restriction map \[\res_{U}: C^{\infty}(X,E)'\rightarrow C_{0}^{\infty}(U,E)'.\] Elements in the image of
$\res_{U}$ are called {\it tempered E-distributions on $U$ with respect to $X$}. When $E$ is the trivial line bundle,
we will omit the symbol $E$ from notations like $C^{\infty}(X,E)$, $C^{\infty}(X,E)'$, etc. For example,
$C_{0}^{\infty}(U)'$ stands for the space of distributions on an open subset $U$.

There is a short exact sequence \[0\rightarrow\ker\res_{U}\rightarrow C^{\infty}(X,E)'\rightarrow\im\res_{U}
\rightarrow 0.\] Put \[Z:=X-U.\] Then, elements in $\ker\res_{U}$ are $E$-distributions on $X$ supported in the closed
subset $Z$. By dualizing, there is a dual exact sequence \[0\rightarrow(\im\res_{U})'\rightarrow C^{\infty}(X,E)
\rightarrow(\ker\res_{U})'\rightarrow 0.\] Put \[\mathcal{S}(U,E)=(\im\res_{U})'\] and call it the {\it (relative)
Schwartz space of sections of $E$ on $U$ with respect to $X$}. Then, $\mathcal{S}(U,E)'=\im\res_{U}$ by duality.
The following lemma gives a characterization of sections in the space $\mathcal{S}(U,E)'=(\im\res_{U})'$.

\begin{lemma}(\cite[Lemma 2.2]{Casselman-Hecht-Milicic})\label{L:2.2}
The subspace $(\im\res_{U})'$ of $C^{\infty}(X,E)$ consists of all global sections in $C^{\infty}(X,E)$ vanishing
with all of their derivatives on $Z$.
\end{lemma}

\begin{proof}
Let $s\in(\im\res_{U})'$. Any distribution supported in a point $x\in Z$ is a finite linear combination of
derivatives of delta like distributions $s\mapsto\alpha(s(x))$ ($\alpha\in E_{x}^{\ast}$). Evaluating these
distributions at $s\in(\im\res_{U})'$, it follows that $s$ vanishes with all derivatives at $x$. Then, all
$s\in(\im\res_{U})'$ vanish with all derivatives on $Z$. The converse follows from
\cite[Lemma 2.1]{Casselman-Hecht-Milicic}, which asserts that any order $\leq m$ $E$-distribution $T$ on $X$
vanishes on $s\in C^{\infty}(X,E)$ such that $Ds|_{\supp(T)}=0$ for all order $\leq m$ smooth differential
operators $D$ on $E$.
\end{proof}

Let $Y$ be a submanifold of $X$ which is also a {\it subanalytic set} (\cite{Bierstone-Milman}) in $X$ (note that
being a submanifold implies that $Y$ is an open subset of $\bar{Y}$). An open subset $U$ containing $Y$ is called
a {\it subanalytic neighborhood} of $Y$ if $U$ is subanalytic in $X$ and $Y$ is a closed subset of $U$, i.e, $\bar{Y}
\cap U=Y$. Let $\mathcal{T}(Y,E)$ denote the space of {\it tempered $E$-distributions on $U$ with support in $Y$}.
By the following Lemma \ref{L:2.6}, the space $\mathcal{T}(Y,E)$ is independent of the choice of the subanalytic
open neighborhood $U$.

\begin{lemma}(\cite[Lemma 2.6]{Casselman-Hecht-Milicic})\label{L:2.6}
If $U'\subset U$ are two subanalytic open neighborhoods of $Y$, then the restriction map $\res_{U,U'}$ induces an
isomorphism of the space of tempered $E$-distributions on $U$ with support in $Y$ onto the space of tempered
$E$-distributions on $U'$ with support in $Y$.
\end{lemma}

\begin{proof}
Injectivity of the restriction map $\res_{U,U'}$ is due to the supporting set condition. The inverse of the
restriction map $\res_{U,U'}$ is given by ``the extension by zero" map.
\end{proof}

For a closed subset $K$ of $X$, let $C^{\infty}(X,E)'_{K}$ denote the space of $E$-distributions on $X$ with support
contained in $K$. Since $Y$ is a closed subset of $U$, we have $\bar{Y}\cap U=Y$. Then, $\bar{Y}-Y\subset X-U=Z$.
Thus, $Y\cup Z$ is a closed subset of $X$. The following lemma is clear.

\begin{lemma}\label{L:tempered3}
The space $\mathcal{T}(Y,E)$ is equal to $C^{\infty}(E)'_{Z\cup Y}/C^{\infty}(E)'_{Z}$.
\end{lemma}

For an integer $p$, let $M_{p}$ denote the closed space of $\mathcal{S}(U,E)$ consisting of sections which vanish
with all derivatives of order $\leq p$ along $Y$. Put $F_{p}\mathcal{T}(Y,E)=M_{p}^{\perp}$ and call it the
{\it space of tempered $E$-distributions with support in $Y$ and of transversal degree $\leq p$}. Then,
\[F\mathcal{T}(Y,E)=\{F_{p}\mathcal{T}(Y,E)\}_{p=0}^{\infty}\] is an exhaustive increasing filtration of
$\mathcal{T}(Y,E)$.


For each $p\in\mathbb{Z}$, let $M_{0p}$ denote the subspace of $M_{p}$ consisting of sections with compact support contained
in $U$. The following lemma gives a useful characterization of $F_{p}\mathcal{T}(Y,E)$ in practice. The proof of Lemma
\ref{L:2.9} is a bit complicated. We refer interested readers to consult the original paper \cite{Casselman-Hecht-Milicic}.

\begin{lemma}(\cite[Lemma 2.9]{Casselman-Hecht-Milicic})\label{L:2.9}
For any integer $p$, we have \[M_{0p}^{\perp}=M_{p}^{\perp}.\]
\end{lemma}

The following is \cite[Example 2.3]{Casselman-Hecht-Milicic}.

\begin{example}(\cite[Example 2.3]{Casselman-Hecht-Milicic})\label{E:linear1}
Let $X=S^{n}$ be an $n$-dimensional sphere, $p\in S^{n}$ and $\pi: S^{n}-\{p\}\rightarrow\mathbb{R}^{n}$ be the
stereographical projection of $S^{n}-\{p\}$ onto $\mathbb{R}^{n}$ with respect to the pole $p$. The closed subspace
of $C^{\infty}(S^{n})$ consisting of smooth functions vanishing with all derivatives at $p$ can be identified with
the ``classical" Schwartz space $\mathcal{S}(\mathbb{R}^{n})$ via the map $\pi$. Dually, this identifies the space
of tempered distributions on $S^{n}-\{p\}$ with respect to $S^{n}$ with the space $\mathcal{S}(\mathbb{R}^{n})'$ of
``classical" tempered distributions on $\mathbb{R}^{n}$.
\end{example}

The following is a generalization of \cite[Example 2.11]{Casselman-Hecht-Milicic} to a general linear subspace.

\begin{example}(\cite[Example 2.11]{Casselman-Hecht-Milicic})\label{E:linear2}
Let $S^{k}$ be a $k$-dimensional sphere in $X=S^{n}$ passing through $p$ such that the stereographical projection image
of $Y=S^{k}-\{p\}$ is equal to the linear subspace $\mathbb{R}^{k}$ of $\mathbb{R}^{n}$ defined by $x_{k+1}=\cdots=
x_{n}=0$. Using Lemma \ref{L:2.9} one can show that \[F_{p}\mathcal{T}(Y)\!=\!\{T\in\mathcal{S}(\mathbb{R}^{n})':
\!(x_{k+1}^{a_{k+1}}\cdots x_{n}^{a_{n}})T=0, a_{j}\geq 0,\!\sum_{k+1\leq j\leq n}\!a_{j}\!=\!p+1\}.\] Then, \begin{equation}\label{Eq:F0}F_{0}\mathcal{T}(Y)=\mathcal{S}(\mathbb{R}^{k})',\end{equation} where
$\mathcal{S}(\mathbb{R}^{k})'$ means the image of the inclusion of $\mathcal{S}(\mathbb{R}^{k})'$ into
$\mathcal{S}(\mathbb{R}^{n})'$. Moreover (cf. \cite[Ch. III, \S 10]{Schwartz}), \begin{equation}\label{Eq:TY}
\mathcal{T}(Y)=\bigoplus_{a_{j}\geq 0}\partial_{k+1}^{a_{k+1}}\cdots\partial_{n}^{a_{n}}F_{0}\mathcal{T}(Y)
\end{equation} and \begin{equation}\label{Eq:Fp}F_{p}\mathcal{T}(Y)=\bigoplus_{a_{j}\geq 0,~\sum a_{j}\leq p}
\partial_{k+1}^{a_{k+1}}\cdots\partial_{n}^{a_{n}}F_{0}\mathcal{T}(Y)\end{equation} for each $p\geq 0$.
\end{example}


\section{Bruhat filtration}\label{S:Bruhat}

In this section we review the Bruhat filtration defined in \cite[\S 3-4]{Casselman-Hecht-Milicic}. The same as in
\S \ref{S:E-distribution}, for most results we only give a sketch of proof. Interested readers are encouraged to
consult the excellent paper \cite{Casselman-Hecht-Milicic} for more complete account and proof.

Let $G$ be a real linear reductive group with Lie algebra $\mathfrak{g}_{0}$ and $K$ be a maximal compact subgroup with Lie
algebra $\mathfrak{k}_{0}$. Let $\mathfrak{g}_0=\mathfrak{k}_0+\mathfrak{q}_0$ be the corresponding Cartan decomposition.
Choose a maximal abelian subspace $\mathfrak{a}_0$ of $\mathfrak{q}_0$. Let $\mathfrak{g}$, $\mathfrak{k}$, $\mathfrak{q}$,
$\mathfrak{a}$ be the complexifications of $\mathfrak{g}_0$, $\mathfrak{k}_0$, $\mathfrak{q}_0$, $\mathfrak{a}_0$
respectively. Write $\Phi=\Phi(\mathfrak{g},\mathfrak{a})$ for the restricted root system from the adjoint action of
$\mathfrak{a}$ on $\mathfrak{g}$. Choose a positive system $\Phi^{+}=\Phi^{+}(\mathfrak{g},\mathfrak{a})$ of $\Phi$. Let
$\mathfrak{n}$ be the subalgebra of $\mathfrak{g}$ spanned by root spaces of positive roots in $\Phi^{+}$, and
$\mathfrak{n}_{0}=\mathfrak{n}\cap\mathfrak{g}_{0}$. There is an Iwasawa decomposition $G=KAN$ where $A$ and $N$ are the
closed Lie subgroups of $G$ corresponding to the Lie subalgebras $\mathfrak{a}_0,\mathfrak{n}_0$ of $\mathfrak{g}_{0}$
respectively. Put $M=Z_{K}(\mathfrak{a}_{0})$ and $P=MAN$. Then, $P$ is a minimal parabolic subgroup of $G$.

Let $X=G/P$ be the flag variety associated to the minimal parabolic subgroup $P$. Write $W=N_{G}(A)/Z_{G}(A)$ and
put $r=|W|$. It is well known that the conjugaction action of $W$ on $\mathfrak{a}$ identifies $W$ with the Weyl
group of the restricted root system $\Phi$. For each $w\in W$, choose an element $\dot{w}\in N_{K}(A)$ representing
$w$ and write $x_{w}=\dot{w}P\in C(w)$. Set \[C(w)=N\cdot x_{w}=N\dot{w}P/P\subset X=G/P,\] which is called the
{\it Bruhat cell attached to $w$} (\cite[line -1]{Casselman-Hecht-Milicic}). Each $C(w)$ is a smooth submanifold and
is a subanalytic set in $X$. The Bruhat decomposition of $G$ relative to $P$ asserts that \[G/P=\bigsqcup_{w\in W}
C(w).\] Let $w_0$ be the longest element in $W$, i.e., the unique element in $W$ which maps $\Phi^{+}$ to $-\Phi^{+}$.
For each $w\in W$, put $N^{w}=wNw^{-1}$. Then, the map \[(N\cap N^{w})\times(N\cap N^{ww_{0}})\rightarrow N,
\quad n_{1}\times n_{2}\mapsto n_{1}n_{2}\] is an isomorphism of algebraic varieties. Hence, $N\cong(N\cap N^{w})
\times(N\cap N^{ww_{0}})$ as an affine algebraic variety. Thus, the map $n\mapsto n\dot{w}P$ ($n\in N\cap N^{ww_{0}}$)
gives an isomorphism $N\cap N^{ww_{0}}\cong C(w)$. Moreover, we know that the complexified Lie algebra $\bar{\mathfrak{n}}
=\mathfrak{n}^{w_0}$ of $\bar{N}:=N^{w_0}$ is spanned by root spaces of negative roots in $\Phi^{-}=-\Phi^{+}$.

Write $W=\{w_{1},\dots,w_{r}\}$ so that $\dim C(w_{i})\leq\dim C(w_{i+1})$ ($1\leq i\leq r-1$). For each $k$
($0\leq k\leq r$), write \[Z_{k}=\bigcup_{1\leq i\leq k}C(w_{i}).\] Here, we remark that the stratification
\[\emptyset=Z_{0}\subset Z_{1}\subset\cdots\subset Z_{r}=X\] is a refinement of the stratification defined in
\cite{Casselman-Hecht-Milicic}. Then, the Buhat filtration defined below is a refinement of that in
\cite{Casselman-Hecht-Milicic}.

Let $(\sigma,V_{\sigma})$ be a finite-dimensional complex linear algebraic representation of $P$. Here, being
algebraic means {\it matrix coefficients} \[c_{\alpha,v}(x)=\langle\alpha, x\cdot v\rangle\ (x\in P)\] of $\sigma$
are regular functions on $P$ (as an algebraic variety) for all $(\alpha,v)\in V_{\sigma}^{\ast}\times V_{\sigma}$.
Different from that in \cite{Casselman-Hecht-Milicic}, we don't assume that the action of $N$ on $V_{\sigma}$ is
trivial. All statements in \cite{Casselman-Hecht-Milicic} extend to our setting. Write \[E(\sigma):=G\times_{P}
V_{\sigma}\] for a smooth equivariant vector bundle with fibre space $V_{\sigma}$ at $P\in G/P$. Write $I(\sigma)=
\Ind_{P}^{G}(\sigma)$ for the {\it un-normalized smooth parabolic induction} from the representation $\sigma$ of
$P$. Then, \[C^{\infty}(E(\sigma))=I(\sigma)\textrm{ and }C^{\infty}(E(\sigma))'=I(\sigma)'.\] For each $k$
($0\leq k\leq r$), let $C^{\infty}_{k}(\sigma)$ be the space of smooth sections of $E(\sigma)$ which vanish with
all derivatives along $Z_{k}$. We call the following finite decreasing filtration of $\mathcal{U}(\mathfrak{g})$
submodules of $I(\sigma)$: \begin{equation*}I(\sigma)=C^{\infty}_0(\sigma)\supset C^{\infty}_{1}(\sigma)\supset
\cdots\supset C^{\infty}_{r-1}(\sigma)\supset C^{\infty}_{r}(\sigma)=\{0\},\end{equation*} the {\it Bruhat
filtration of $I(\sigma)$}. Write \[\mathrm{gr}I(\sigma)=\sum_{k\geq 0}C^{\infty}_{k}(\sigma)/
C^{\infty}_{k+1}(\sigma)\] for the graded module corresponding to the Bruhat filtration of $I(\sigma)$. Let \[I_{k,\sigma}=C^{\infty}_{k}(\sigma)^{\perp}\] be the space of $E(\sigma)$-distributions vanishing on all
sections in $C^{\infty}_{k}(\sigma)$. Then, $I_{0,\sigma}=0$ and $I_{r,\sigma}=I(\sigma)'$.

For each $w\in W$, choose a subanalytic open neighborhood $U$ of $C(w)$ in $X$. Let $\mathcal{T}(C(w),E(\sigma))$
be the space of {\it tempered $E(\sigma)$-distributions on $U$ with support in $C(w)$} and $J(w,\sigma)$ be the
strong dual of $\mathcal{T}(C(w),E(\sigma))$. By Lemma \ref{L:2.6}, $\mathcal{T}(C(w),E(\sigma))$ and $J(w,\sigma)$
do not depend on the choice of the subanalytic open neighborhood $U$. Put \[V^{w}=N^{ww_{0}}\dot{w}P/P.\] As shown
in \cite[p. 166]{Casselman-Hecht-Milicic} and \cite[Lemma 2.4]{Abe}, $V^{w}$ is a subanalytic open neighborhood of
$C(w)$ in $X$. Hence, we could take $U=V^{w}$ while defining $\mathcal{T}(C(w),E(\sigma))$ and $J(w,\sigma)$.

\begin{lemma}\label{L:Bruhat2}
For each $k$ ($1\leq k\leq r$), the space $I_{k,\sigma}=(C^{\infty}_{k}(\sigma))^{\perp}$ is equal to the space
of $E(\sigma)$-distributions supported in $Z_{k}$. There is an exact sequence
\[0\rightarrow C^{\infty}_{k-1}(\sigma)^{\perp}\rightarrow C^{\infty}_{k}(\sigma)^{\perp}\rightarrow
\mathcal{T}(C(w_{k}),E(\sigma))\rightarrow 0.\]
\end{lemma}

\begin{proof}
The first statement follows from Lemma \ref{L:2.2} by dualizing. The second statement is a special case of Lemma
\ref{L:tempered3}.
\end{proof}

By Lemma \ref{L:Bruhat2}, we have \[I_{k,\sigma}/I_{k-1,\sigma}=\mathcal{T}(C(w_{k}),E(\sigma)),\ \forall k,
1\leq k\leq r.\]

\begin{theorem}(\cite[Theorem 4.1]{Casselman-Hecht-Milicic})\label{T:4.1}
For each $k$ ($1\leq k\leq r$), there is a short exact sequence \begin{equation*}
0\longrightarrow C^{\infty}_{k}(\sigma)\longrightarrow C^{\infty}_{k-1}(\sigma)\longrightarrow J(w_{k},\sigma)
\longrightarrow 0.\end{equation*} There is an isomorphism of $\mathcal{U}(\mathfrak{g})$-modules:
\begin{equation*}
\mathrm{gr}I(\sigma)\cong\bigoplus_{k\in\mathbb{Z}_{1\leq k\leq r}}J(w_{k},\sigma).
\end{equation*}
\end{theorem}

\begin{proof}
By dualizing, the first statement follows from Lemma \ref{L:Bruhat2}. The second statement is a consequence of the
first one.
\end{proof}

As in \S \ref{S:E-distribution}, let $M_{p,w,\sigma}$ denote the subspace of $\mathcal{S}(V^{w},E(\sigma))$
consisting of sections which vanish with all derivatives of order $\leq p$ along $C(w)$. Then,
\[F_{p}\mathcal{T}(C(w),E(\sigma))=M_{p,w,\sigma}^{\perp}\ (\forall p\in\mathbb{Z}).\] Put
\[\Gr^{p}\mathcal{T}(C(w),E(\sigma))=F_{p}\mathcal{T}(C(w),E(\sigma))/F_{p-1}\mathcal{T}(C(w),E(\sigma)).\]

\begin{lemma}(\cite[Lemma 3.5]{Casselman-Hecht-Milicic})\label{L:3.5}
For any $p\in\mathbb{Z}$, the subspace $F_{p}\mathcal{T}(C(w),E(\sigma))$ of $\mathcal{T}(C(w),E(\sigma))$ is
$\mathcal{U}(\mathfrak{p})$-invariant.
\end{lemma}

For $Y\in\mathfrak{g}_{0}$, define a differential operator $\tilde{L}_{Y}$ on $C^{\infty}(X,E(\sigma))$ by setting
\[(\tilde{L}_{Y}\varphi)(g)=\frac{d}{dt}\varphi(\exp(-tY)g)|_{t=0}\] for any $\varphi\in C^{\infty}(X,E(\sigma))$
and any $g\in G$, where $\varphi$ is regarded as a $V_{\sigma}$-valued function on $G$. We have \[\tilde{L}_{[Y_{1},
Y_{2}]}=\tilde{L}_{Y_{1}}\tilde{L}_{Y_{2}}-\tilde{L}_{Y_{2}}\tilde{L}_{Y_{1}},\ \forall Y_{1},Y_{2}\in
\mathfrak{g}_{0}.\] Then, $\tilde{L}_{\cdot}$ extends to an action of $\mathcal{U}(\mathfrak{g})$ on
$C^{\infty}(X,E(\sigma))$ by setting \[\tilde{L}(\lambda Y_{1}\cdots Y_{k})\varphi=\lambda\tilde{L}_{Y_{1}}(\cdots
(\tilde{L}_{Y_{k}}\varphi)\cdots)\] for any $Y_{1},\dots,Y_{k}\in\mathfrak{g}_{0}$, $\lambda\in\mathbb{C}$ and
$\varphi\in C^{\infty}(X,E(\sigma))$. For any $Y_{1},\dots,Y_{k}\in\mathfrak{g}$, define \[(Y_{1}\cdots Y_{k})^{t}
=(-1)^{k}Y_{k}\cdots Y_{1}\] and extends it linearly to all elements in $\mathcal{U}(\mathfrak{g})$. Then,
$Y\mapsto Y^{t}$ ($Y\in\mathcal{U}(\mathfrak{g})$) is an anti-automorphism of $\mathcal{U}(\mathfrak{g})$. For an
element $Y\in\mathcal{U}(\mathfrak{g})$ and a distribution $T\in C^{\infty}(X,E(\sigma))'$, define a distribution
$\tilde{L}_{Y}(T)\in C^{\infty}(X,E(\sigma))'$ by setting \[\langle\tilde{L}_{Y}(T),\varphi\rangle=\langle T,
\tilde{L}_{Y^{t}}(\varphi)\rangle,\ \forall\varphi\in C^{\infty}(X,E(\sigma)).\]

\begin{lemma}\label{L:Abe-A.3}
We have the following assertions:
\begin{enumerate}
\item[(i)]For each $w\in W$, we have \[F_{0}\mathcal{T}(C(w),E(\sigma))=\mathcal{S}(C(w),E(\sigma)|_{C(w)})'.\]
\item[(ii)]For each $p\geq 0$, the map $(Y,T)\mapsto\tilde{L}_{Y}(T)$ gives an isomorphism
\[\mathcal{U}_{p}(\mathfrak{n}^{ww_{0}}\cap\bar{\mathfrak{n}})\otimes_{\mathbb{C}}F_{0}\mathcal{T}(C(w),E(\sigma))
\cong F_{p}\mathcal{T}(C(w),E(\sigma)).\]
\item[(iii)]The map $(Y,T)\mapsto\tilde{L}_{Y}(T)$ gives an isomorphism \[\mathcal{U}(\mathfrak{n}^{ww_{0}})
\otimes_{\mathcal{U}(\mathfrak{n}^{ww_{0}}\cap\mathfrak{n})}F_{0}\mathcal{T}(C(w),E(\sigma))\cong
\mathcal{T}(C(w),E(\sigma)).\]
\end{enumerate}
\end{lemma}

\begin{proof}
The equality in (i) follows from \eqref{Eq:F0}. The equality in (ii) follows form \eqref{Eq:Fp}. Since \[F_{0}\mathcal{T}
(C(w),E(\sigma))\otimes_{\mathcal{U}(\mathfrak{n}^{ww_{0}}\cap\mathfrak{n})}\mathcal{U}(\mathfrak{n}^{ww_{0}})
\cong F_{0}\mathcal{T}(C(w),E(\sigma))\otimes_{\mathbb{C}}\mathcal{U}(\mathfrak{n}^{ww_{0}}\cap\bar{\mathfrak{n}}),\]
the equality in (iii) follows from \eqref{Eq:TY}.
\end{proof}

Let $J$ be the ideal of polynomials in $R(V^{w})$ vanishing along $C(w)$. From the multiplication map
\[J\otimes_{R(V^{w})}F_{p}\mathcal{T}(C(w),E(\sigma))\rightarrow F_{p-1}\mathcal{T}(C(w),E(\sigma)),\] we get \[J^{p}/J^{p+1}\otimes_{R(C(w)))}\Gr^{p}\mathcal{T}(C(w),E(\sigma))\rightarrow F_{0}\mathcal{T}(C(w),E(\sigma)).\] Put \[L_{p}(w):=\mathrm{Hom}_{R(C(w))}(J^{p}/J^{p+1},R(C(w))).\] Notice that $J^{p}/J^{p+1}$ is a free $R(C(w))$-module of
finite rank, and so is $L_{p}(w)$. Then, we get a $\mathcal{U}(\mathfrak{p})$-module homomorphism \[\alpha_{p}:
\Gr^{p}\mathcal{T}(C(w),E(\sigma))\rightarrow L_{p}(w)\otimes_{R(C(w))}F_{0}\mathcal{T}(C(w),E(\sigma)),\]
which is clearly an isomorphism. Then, we get the following Lemma \ref{L:4.3}.

\begin{lemma}(\cite[Lemma 4.3]{Casselman-Hecht-Milicic})\label{L:4.3}
As $\mathcal{U}(\mathfrak{p})$ modules, we have \begin{eqnarray*}&&\Gr^{p}\mathcal{T}(C(w),E(\sigma))\\&\cong&L_{p}(w)
\otimes_{R(C(w))}F_{0}\mathcal{T}(C(w),E(\sigma))\\&\cong&L_{p}(w)\otimes_{R(C(w))}\mathcal{S}(C(w),E(\sigma))'.
\end{eqnarray*}
\end{lemma}

\begin{lemma}(\cite[Lemma 4.4]{Casselman-Hecht-Milicic})\label{L:4.4}
The $R(C(w))$ module $L_{p}(w)$ admits a finite increasing filtration by $N$-equivariant $R(C(w))$-submodules
such that each graded piece is isomorphic to $R(C(w))$ as an $N$-equivariant $R(C(w))$ module.
\end{lemma}

\begin{proof}
The inclusion $C(w)\subset V^{w}$ can be identified with $N\cap N^{ww_{0}}\subset N^{ww_{0}}$. Taking exponential
map this is equivalent to $\mathfrak{n}\cap\mathfrak{n}^{ww_{0}}\subset\mathfrak{n}^{ww_{0}}$, which is a linear
subspace in a real linear space. Hence, $L_{p}(w)$ is the space of global regular sections of a finite rank free
$\mathcal{O}_{C(w)}$-module $\mathcal{V}_{p}$. This $\mathcal{O}_{C(w)}$-module is a sheaf $\mathcal{V}_{p}$ of
local sections of an $N$-homogeneous algebraic vector bundle $V_{p}$ on $C(w)$. Notice that $N$ is a unipotent
algebraic group. Then, so is $\Stab_{N}(x_{w})=N\cap wNw^{-1}$. Then, the geometric fibre of $V_{p}$ at $x_{w}$
as a representation of $\Stab_{N}(x_{w})$ admits a finite increasing filtration with each graded piece a trivial
representation of $\Stab_{N}(x_{w})$. Then, the $N$-homogeneous algebraic vector bundle $V_{p}$ admits a finite
increasing filtration with each graded piece isomorphic to the trivial bundle on $C(w)$ and with trivial $N$
action. Since $C(w)$ is affine, by Serre's theorem the global section functor is exact. Therefore, $L_{p}(w)$
admits a finite increasing filtration with each graded piece being isomorphic to $R(C(w))$ as an $N$-equivariant
$R(C(w))$ module.
\end{proof}

The following corollary is a direct consequence of Lemmas \ref{L:4.3} and \ref{L:4.4}.

\begin{corollary}(\cite[Lemma 4.3]{Casselman-Hecht-Milicic})\label{C:free}
For each $p\geq 0$, $\Gr^{p}\mathcal{T}(C(w),E(\sigma))$ admits a finite increasing filtration such that each graded
piece is isomorphic to $\mathcal{S}(C(w),E(\sigma))'$ as a $\mathcal{U}(\mathfrak{n})$ module.
\end{corollary}

\section{Casselman-Jacquet modules}\label{S:Abe}

Let $\mathfrak{n'}$ be a nilpotent complex Lie algebra and $U$ be a complex linear representation of
$\mathcal{U}(\mathfrak{n'})$. We define \[U^{[\mathfrak{n'}]}=\{v\in U:\exists k\in\mathbb{Z}_{>0},
\mathfrak{n'}^{k}\cdot v=0\}\] and call it {\it the space of nilpotnet elements} in $U$.

Write $\eta:\mathcal{U}(\mathfrak{n'})\rightarrow\mathbb{C}$ for the augmented homomorphism of
$\mathcal{U}(\mathfrak{n'})$ defined by $\eta(X_{1}\cdots X_{q})=0$ for any $X_{1},\dots,X_{q}\in\mathfrak{n'}$
whenever $q>0$. Then, \[\ker\eta=\mathfrak{n'}\mathcal{U}(\mathfrak{n'})=\mathcal{U}(\mathfrak{n'})\mathfrak{n'}\]
and it is the augmented ideal of $\mathcal{U}(\mathfrak{n'})$.

\begin{lemma}\label{L:finite1}
For any $v\in V$, the following conditions are equivalent: \begin{enumerate}
\item[(1)]For some $k\geq 1$, $(\ker\eta)^{k}v=0$.
\item[(2)]For some $k\geq 1$, $\mathfrak{n'}^{k}\cdot v=0$.
\end{enumerate}
\end{lemma}

\begin{proof}
$(2)\Rightarrow(1)$. Suppose $\mathfrak{n'}^{k}\cdot v=0$. Then, $(\mathcal{U}(\mathfrak{n'})\mathfrak{n'}^{k})
\cdot v=0$. Since $\mathcal{U}(\mathfrak{n'})=\mathcal{U}_{k}(\mathfrak{n'})+\mathcal{U}(\mathfrak{n'})
\mathfrak{n'}^{k}$ and $\dim\mathcal{U}_{k}(\mathfrak{n'})<\infty$, we get $\dim\mathcal{U}(\mathfrak{n'})v<
\infty$. Put $U:=\mathcal{U}(\mathfrak{n'})v$. Then, $U$ is a finite-dimensional representation of
$\mathcal{U}(\mathfrak{n'})$ and each $X\in\mathfrak{n'}$ acts on $U$ as a nilpotent endomorphism. By Engel's
theorem, there is a filtration $U=U_{0}\supset U_{1}\supset\cdots\supset U_{k}=0$ of $U$ such that
$\mathfrak{n'}U_{i}\subset U_{i-1}$ ($1\leq i\leq k$). Then, $(\ker\eta)U_{i}\subset U_{i-1}$ ($1\leq i\leq k$).
Thus, $(\ker\eta)^{k}v=0$.

$(1)\Rightarrow(2)$. It follows from $\mathfrak{n'}\subset\ker\eta$.
\end{proof}

Alternatively, from the fact $\mathfrak{n'}\mathcal{U}(\mathfrak{n'})=\mathcal{U}(\mathfrak{n'})\mathfrak{n'}$,
we get $\ker\eta^{k}=(\mathcal{U}(\mathfrak{n'})\mathfrak{n'})^{k}=\mathcal{U}(\mathfrak{n'})\mathfrak{n'}^{k}$.
Then, if follows that $(1)\Leftrightarrow(2)$ in Lemma \ref{L:finite1}.

\begin{lemma}\label{L:finite2}
Let $V,U$ be two complex linear representations of $\mathcal{U}(\mathfrak{n'})$. Then we have the following assertions:
\begin{enumerate}
\item[(i)]$V^{[\mathfrak{n'}]}\otimes U^{[\mathfrak{n'}]}\subset(V\otimes U)^{[\mathfrak{n'}]}$.
\item[(ii)]If $V=V^{[\mathfrak{n'}]}$, then $(V\otimes U)^{[\mathfrak{n'}]}=V\otimes U^{[\mathfrak{n'}]}$.
\item[(iii)]If $U=U^{[\mathfrak{n'}]}$, then $(V\otimes U)^{[\mathfrak{n'}]}=V^{[\mathfrak{n'}]}\otimes U$.
\end{enumerate}
\end{lemma}

\begin{proof}
(i) is obvious. (ii) and (iii) are similar. We show (ii) below.

Suppose that $V=V^{[\mathfrak{n'}]}$. For each $k\geq 1$, put \[V_{k}=\{v\in V:(\ker\eta)^{k}v=0\}.\] Then,
$\mathfrak{n'}V_{k}\subset V_{k-1}$ for each $k\geq 1$. By Lemma \ref{L:finite1}, we have \[V_{k}=\{v\in V:
\mathfrak{n'}^{k}v=0\}.\] Then, \[0=V_{0}\subset V_{1}\subset\cdots\subset V_{k}\subset\cdots\] is an exhaustive
ascending filtration of $V$ by the assumption $V=V^{[\mathfrak{n'}]}$. Let $x$ be an element in
$(V\otimes U)^{[\mathfrak{n'}]}$. From the above, we can find a set of linearly independent elements
$\{v_{i}:1\leq i\leq k\}$ of $V$ such that \[(\ker\eta)v_{i}\subset\span\{v_{j}:1\leq j\leq i-1\}\] for each $i$
($1\leq i\leq k$) and $x\in\span\{v_{j}:1\leq j\leq k\}\otimes U$. Write \[x=\sum_{1\leq j\leq k}v_{j}
\otimes u_{j}\] where $u_{j}\in U$ ($1\leq j\leq k$). Projecting to \[(\span\{v_{j}:1\leq j\leq k\}/
\span\{v_{j}:1\leq j\leq k-1\})\otimes U,\] we get that $v_{k}\otimes u_{k}$ is annihilated by a power of
$\ker\eta$ in this quotient module of $\span\{v_{j}:1\leq j\leq k\}\otimes U$. Since
$\span\{v_{j}:1\leq j\leq k\}/\span\{v_{j}:1\leq j\leq k-1\}$ is isomorphic to the trivial representation of
$\mathcal{U}(\mathfrak{n'})$, we get $u_{k}\in U^{[\mathfrak{n'}]}$. Then by (i), we have
$v_{k}\otimes u_{k}\in(V\otimes U)^{[\mathfrak{n'}]}$. Thus, \[\sum_{1\leq j\leq k-1}v_{j}\otimes u_{j}\in
(V\otimes U)^{[\mathfrak{n'}]}.\] Inductively, we show that $u_{k-1},\dots,u_{1}$ are all in
$U^{[\mathfrak{n'}]}$. Hence, $x\in V\otimes U^{[\mathfrak{n'}]}$.
\end{proof}

For a $\mathcal{U}(\mathfrak{p})$-module $V$, we call $V^{[\mathfrak{n}]}$ the {\it Casselman-Jacquet module} of $V$.
Then, $V^{[\mathfrak{n}]}$ is still a $\mathcal{U}(\mathfrak{p})$-module, and $\mathfrak{n}$ acts on it locally
nilpotently in the sense that: for any $v\in V^{[\mathfrak{n}]}$, $\mathcal{U}(\mathfrak{n})\cdot v$ is a
finite-dimensional space. Moreover, if $V$ is a $\mathcal{U}(\mathfrak{g})$-module, then so is $V^{[\mathfrak{n}]}$
(\cite{Casselman2}). It is clear that the functor $V\rightarrow V^{[\mathfrak{n}]}$ is left exact. However, it is not
right exact in general. The goal of this section is to show that the fuctor $V\rightarrow V^{[\mathfrak{n}]}$ preserves
exactness of some short exact sequences arising in \S \ref{S:Bruhat}.

\smallskip

Let $U$ be a Zariski closed subgroup of $N$. In \cite[p. 169]{Casselman-Hecht-Milicic} there is a definition of
polynomials on the affine space $N/U$: a smooth function $f$ on $N/U$ is called a {\it polynomial} if its annihilator
in $\mathcal{U}(\mathfrak{n})$ contains $\mathfrak{n}^{k}$ for some $k\geq 1$. The notions of polynomial and regular
function coincide in this case, as the following Lemma \ref{L:Abe0} shows. Lemma \ref{L:Abe0} is an adjustment of
\cite[Proposition A.4]{Abe}. The same proof as that in \cite{Abe} works. Define a distribution $\delta_{N/U}$ on $N/U$
by \[\delta_{N/U}(f)=\int_{N/U}f(x)\d x\] for any $f\in C_{c}^{\infty}(N/U)$, where $\d x$ is a fixed $N$ invariant
measure on $N/U$.

\begin{lemma}(\cite[Proposition A.4]{Abe})\label{L:Abe0}
We have the following assertions.
\begin{enumerate}
\item[(i)]For any $k\geq 1$, there exists $l\geq 1$ such that if a distribution $T$ on $N/U$ satisfies
$\mathfrak{n}^{k}T=0$, then $T\in R_{l}(N/U)\delta_{N/U}$.
\item[(ii)]For any $l\geq 1$, there exists $k\geq 1$ such that $\mathfrak{n}^{k}(R_{l}(N/U)\delta_{N/U})=0$.
\end{enumerate}
\end{lemma}

Let $w\in W$. For $f\in R(C(w))$ and $u'\in V_{\sigma}^{\ast}$, define a distribution $(f\otimes u')\delta_{C(w)}
\in\mathcal{S}(C(w),E(\sigma))'$ by setting \[\langle(f\otimes u')\delta_{C(w)},g\rangle=\int_{N\cap N^{ww_{0}}}
f(y\dot{w}P)u'(g(y))\d y\] for all $g\in\mathcal{S}(C(w),E(\sigma))$, where $\d y$ is a volume form on
$N\cap N^{ww_{0}}$ which induces an $N$ invariant measure $\d x$ on $C(w)=N/N\cap N^{w}\cong N \cap N^{ww_{0}}$
and $g$ is regarded as a function on $N$. Choose a basis $\{v_{1},\dots,v_{m}\}$ of $V_{\sigma}$ and let
$\{\alpha_{1},\dots,\alpha_{m}\}$ be the dual basis of $V_{\sigma}^{\ast}$. For any $v\in V_{\sigma}$ and
$\alpha\in V_{\sigma}^{\ast}$, define the {\it matrix coefficient} \[\tilde{c}_{\alpha,v}(x)=
\langle\alpha, x^{-1}\cdot v\rangle=c_{\alpha,v}(x^{-1}),\ \forall x\in P.\]

\begin{lemma}\label{L:Abe1}
For each $w\in W$, we have \[\mathcal{S}(C(w),E(\sigma))'^{[\mathfrak{n}]}=(R(C(w))\otimes V_{\sigma}^{\ast})
\delta_{C(w)}.\]
\end{lemma}

\begin{proof}
Due to the fact that $N$ is a unipotent group and $V_{\sigma}$ is finite-dimensional algebraic representation of
$P$, we have $V_{\sigma}^{[\mathfrak{n}\cap\mathfrak{n}^{ww_{0}}]}=V_{\sigma}$. Note that $C(w)\cong N\cap N^{ww_{0}}$
and $E(\sigma)|_{C(w)}$ is a trivial $N\cap N^{ww_{0}}$-equivariant bundle. Then, \[\mathcal{S}(C(w),
E(\sigma))'^{[\mathfrak{n}]}\subset\mathcal{S}(C(w),E(\sigma))'^{[\mathfrak{n}\cap\mathfrak{n}^{ww_{0}}]}=
\mathcal{S}(C(w)))'^{[\mathfrak{n}\cap\mathfrak{n}^{ww_{0}}]}\otimes V_{\sigma}^{\ast}\] by Lemma
\ref{L:finite2} (iii) and \[\mathcal{S}(C(w))^{'[\mathfrak{n}\cap\mathfrak{n}^{ww_{0}}]}=R(C(w))\delta_{C(w)}\]
by Lemma \ref{L:Abe0}. Thus, \begin{equation}\label{Eq:Abe1-a}\mathcal{S}(C(w),E(\sigma))'^{[\mathfrak{n}]}\subset
(R(C(w))\delta_{C(w)})\otimes V_{\sigma}^{\ast}=(R(C(w))\otimes V_{\sigma}^{\ast})\delta_{C(w)}.\end{equation}

Define a representation $(\dot{w}\sigma,V_{\dot{w}\sigma})$ of $N\cap N^{w}$ whose underlying space is $V_{\sigma}$
and the group action is given by \[(\dot{w}\sigma)(x)=\sigma(\dot{w}^{-1}x\dot{w}),\forall x\in N
\cap N^{w}.\] Then, we have \[E(\sigma)|_{C(w)}\cong N\times_{N\cap N^{w}}V_{\dot{w}\sigma}\] via the
identification $C(w)\cong N/N\cap N^{w}$. Recall that a section of the dual bundle $N\times_{N\cap N^{w}}
V_{\dot{w}\sigma}^{\ast}$ may be identified with a function $h: N\rightarrow V_{\dot{w}\sigma}^{\ast}$ such
that \[h(nn')=(\dot{w}\sigma^{\ast}(n'^{-1}))h(n),\ \forall(n,n')\in N\times (N\cap N^{w}).\] Write
$R_{\dot{w}\sigma}$ for the space of such sections $h: N\rightarrow V_{\dot{w}\sigma}^{\ast}$ such that
$\langle h(x),v\rangle$ is a regular function on $N$ for any fixed vector $v\in V_{\sigma}=V_{\dot{w}\sigma}$. We
may identify a section $g\in\mathcal{S}(C(w),E(\sigma)|_{C(w)})$ with a function $g:N\rightarrow V_{\dot{w}\sigma}$
such that \[g(nn')=(\dot{w}\sigma(n'^{-1}))g(n),\ \forall(n,n')\in N\times (N\cap N^{w})\] and
$\langle\alpha,g(x)\rangle$ is a Schwartz function on $N\cap N^{ww_{0}}$ for any fixed vector
$\alpha\in V_{\sigma}^{\ast}=V_{\dot{w}\sigma}^{\ast}$. For any section $h\in R_{\dot{w}\sigma}$ and any section $g\in\mathcal{S}(C(w),E(\sigma)|_{C(w)})$, define \[(h\otimes g)(x)=\langle h(x),g(x)\rangle,\forall x\in N.\]
where $\langle\cdot,\cdot\rangle$ is the pairing $V_{\sigma}^{\ast}\times V_{\sigma}\rightarrow\mathbb{C}$.
Then, \[(h\otimes g)(nn')=(h\otimes g)(n),\forall (n,n')\in N\times(N\cap N^{w}).\] Thus, $h\otimes g$
represents a function on $C(w)\cong N/(N\cap N^{w})$ and it is a Schwartz function. Hence, we can define
\[\langle h,g\rangle=\int_{C(w)}(h\otimes g)(x)\d x=\int_{N\cap N^{ww_{0}}}(h\otimes g)(y)\d y.\]

For any $f\in R(C(w))$ and $u'\in V_{\sigma}^{\ast}$, define $f\otimes u': N\rightarrow V_{\dot{w}\sigma}$ by
setting \[(f\otimes u')(nn')=f(n\dot{w}P)(\dot{w}\sigma^{\ast}(n'^{-1})u'),\ \forall (n,n')\in(N\cap N^{ww_{0}})
\times(N\cap N^{w}).\] Then, $f\otimes u'\in R_{\dot{w}\sigma}$. Due to the fact that $C(w)\cong N/(N\cap N^{w})
\cong N\cap N^{ww_{0}}$, the map \[(f\otimes u')\delta_{C(w)}\mapsto f\otimes u',\ f\in R(C(w)),
u'\in V_{\sigma}^{\ast}\] gives a bijection \[(R(C(w))\otimes V_{\sigma}^{\ast})\delta_{C(w)}
\cong R_{\dot{w}\sigma}.\] Moreover, we have \[\langle(f\otimes u')\delta_{C(w)},g\rangle=\langle f\otimes u',
g\rangle\] for any $f\in R(C(w))$, any $u'\in V_{\sigma}^{\ast}$ and any $g\in\mathcal{S}(C(w),E(\sigma)|_{C(w)})$.
Therefore, the $\mathcal{U}(\mathfrak{n})$ representations $(R(C(w))\otimes V_{\sigma}^{\ast})\delta_{C(w)}$
and $R_{\dot{w}\sigma}$ are isomorphic.

When $N$ acts trivially on $V_{\sigma}$, $R_{\dot{w}\sigma}$ is isomorphic to the direct sum of $\dim V_{\sigma}$
copies of $R(C(w))$. By Lemma \ref{L:Abe0} (ii), each section $h\in R_{\dot{w}\sigma}$ is annihilated by
$\mathfrak{n}^{k}$ for some $k\geq 1$. Then, \[(R(C(w))\otimes V_{\sigma}^{\ast})\delta_{C(w)}\subset
\mathcal{S}(C(w),E(\sigma))'^{[\mathfrak{n}]}.\] In general, take a composition series \[0=V_{0}\subset V_{1}
\subset\cdots\subset V_{s}=V_{\sigma}\] of $V_{\sigma}$ and write $\sigma_{i}$ for the representation of $P$
on $V_{i}/V_{i-1}$ ($1\leq i\leq s$). Accordingly, we have a filtration \[0=R_{\dot{w}\sigma_{0}}
\subset R_{\dot{w}\sigma_{1}}\subset\cdots\subset R_{\dot{w}\sigma_{s}}=R_{\dot{w}\sigma}.\] For each $i$
($1\leq i\leq s$), since $V_{i}/V_{i-1}$ is an irreducible algebraic representation of $P$, then $N$ acts
trivially on it. Hence, any element of $ R_{\dot{w}\sigma_{i}}/ R_{\dot{w}\sigma_{i-1}}
\cong R_{\dot{w}(\sigma_{i}/\sigma_{i-1})}$ is annihilated by $\mathfrak{n}^{k}$  for some $k\geq 1$.
Then, any element of $R_{\dot{w}\sigma}$ is annihilated by $\mathfrak{n}^{k}$  for some $k\geq 1$.
Thus, \begin{equation}\label{Eq:Abe1-b}(R(C(w))\otimes V_{\sigma}^{\ast})\delta_{C(w)}\subset
\mathcal{S}(C(w),E(\sigma))'^{[\mathfrak{n}]}.\end{equation} Combining \eqref{Eq:Abe1-a} and
\eqref{Eq:Abe1-b}, we get \[\mathcal{S}(C(w),E(\sigma))'^{[\mathfrak{n}]}=(R(C(w))\otimes V_{\sigma}^{\ast})
\delta_{C(w)}.\]
\end{proof}

Set \begin{equation}\label{Jw}J_{w,\sigma}\!=\!\span\{\tilde{L}_{Y}\!((f\!\otimes\!u')
\delta_{C(w)}):\!Y\!\in\!\mathcal{U}(\mathfrak{n}^{ww_{0}}\cap\bar{\mathfrak{n}}),f\!\in\!R(C(w)),
u'\!\in\!V_{\sigma}^{\ast}\}.\end{equation}

The following Lemma \ref{L:Abe4} follows from results in \cite[\S 2 and \S 3]{Abe}. Since Lemma \ref{L:Abe-Gp}
below does not directly follow from results in \cite{Abe}, we give a complete proof for Lemma \ref{L:Abe4}
which can be adjusted to show Lemma \ref{L:Abe-Gp}.

\begin{lemma}\label{L:Abe4}
For each $w\in W$, we have \[\mathcal{T}(C(w),E(\sigma))^{[\mathfrak{n}]}=J_{w,\sigma}.\]
\end{lemma}

\begin{proof}
Put \begin{equation}\label{Jw2}J'_{w,\sigma}=\span\{\tilde{L}_{Y}((f\otimes u')\delta_{C(w)}):Y\in\mathcal{U}
(\mathfrak{g}),f\in R(C(w)),u'\in V_{\sigma}^{\ast}\}.\end{equation} Then, it is clear that \begin{equation}
\label{Eq:Abe4-1}J_{w,\sigma}\subset J'_{w,\sigma}.\end{equation} By Lemma \ref{L:Abe1}, $J'_{w,\sigma}$ is a
quotient of $\mathcal{U}(\mathfrak{g})\otimes_{\mathbb{C}}F_{0}\mathcal{T}(C(w),E(\sigma))'^{[\mathfrak{n}]}$
as a representation of $\mathcal{U}(\mathfrak{n})$. Since each element of $\mathcal{U}(\mathfrak{g})$ (or $F_{0}\mathcal{T}(C(w),E(\sigma))'^{[\mathfrak{n}]}$) is annihilated by a power of $\mathfrak{n}$, by Lemma
\ref{L:finite2} (i) we have \begin{equation}\label{Eq:Abe4-2}J'_{w,\sigma}\subset\mathcal{T}(C(w),
E(\sigma))^{[\mathfrak{n}]}.\end{equation}

By Lemma \ref{L:Abe-A.3} (ii), we have \[\mathcal{T}(C(w),E(\sigma))\cong\mathcal{U}(\mathfrak{n}^{ww_{0}}\cap
\bar{\mathfrak{n}})\otimes_{\mathbb{C}}F_{0}\mathcal{T}(C(w),E(\sigma)).\] Take an element $H\in\mathfrak{a}_{0}$
such that $\alpha(H)\in\mathbb{Z}_{>0}$ for each root $\alpha\in\Phi^{+}(\mathfrak{g},\mathfrak{a})$. The
conjugation action of $H$ on $\mathcal{U}(\mathfrak{g})$ gives a grading \[\mathcal{U}(\mathfrak{g})=
\sum_{k\in\mathbb{Z}}\mathcal{U}(\mathfrak{g})[k]\] defined by \[\mathcal{U}(\mathfrak{g})[k]=
\{Y\in\mathcal{U}(\mathfrak{g}):HY-YH=kY\}.\] Put \[\mathcal{U}(\mathfrak{n}^{ww_{0}}\cap\bar{\mathfrak{n}})[k]
=\mathcal{U}(\mathfrak{n}^{ww_{0}}\cap\bar{\mathfrak{n}})\cap\mathcal{U}(\mathfrak{g})[k].\] Since $H$ normalizes
$\mathfrak{n}^{ww_{0}}\cap\bar{\mathfrak{n}}$ and $\bar{\mathfrak{n}}$ is the sum of $H$-eigenspaces of negative
eigenvalues, we have a grading \[\mathcal{U}(\mathfrak{n}^{ww_{0}}\cap\bar{\mathfrak{n}})=\sum_{k\in\mathbb{Z}_{\geq 0}}\mathcal{U}(\mathfrak{n}^{ww_{0}}\cap\bar{\mathfrak{n}})[-k].\] For each $k\in\mathbb{Z}_{\geq 0}$, put \begin{equation}\label{Eq:Jw3}J_{w,\sigma,k}\!=\!\span\{\tilde{L}_{Y}T:\!Y\!\in\!\mathcal{U}(\mathfrak{n}^{ww_{0}}\!
\cap\bar{\mathfrak{n}})[-k'],k'\leq k,T\!\in\!F_{0}\mathcal{T}(C(w),E(\sigma))\}.\end{equation}
Then, \[\mathbb{C}=J_{w,\sigma,0}\subset\cdots J_{w,\sigma,k}\subset\cdots\] form an exhaustive ascending
filtration of $\mathcal{T}(C(w),E(\sigma))$.

We show that: for each $k\geq 0$, \begin{equation}\label{Eq:YT}Y(\tilde{L}_{Y'}Y)\in\tilde{L}_{Y'}(YT)+J_{w,\sigma,
k-1}\end{equation} for any $Y\in\mathfrak{n}^{ww_{0}}\cap\mathfrak{n}$, $Y'\in\mathcal{U}(\mathfrak{n}^{ww_{0}}
\cap\bar{\mathfrak{n}})[k]$ and $T\in F_{0}\mathcal{T}(C(w),E(\sigma)$. Prove by induction on $k$. When $k=0$,
this is clear. Let $k_{0}\geq 1$ and suppose \eqref{Eq:YT} holds when $k<k_{0}$. When $k=k_{0}$, we have
\begin{equation}\label{Eq:Abe4-3}Y(\tilde{L}_{Y_{1}\cdots Y_{l}}T)=\sum_{1\leq i\leq l}(Y_{1}\cdots Y_{i-1}
[Y,Y_{i}])\tilde{L}_{Y_{i+1}\cdots Y_{l}}T+(Y_{1}\cdots Y_{l}Y)T.\end{equation} for any $Y_{1},\dots,Y_{l}\in
\mathfrak{n}^{ww_{0}}\cap\bar{\mathfrak{n}}$, $Y\in\mathfrak{n}^{ww_{0}}\cap\mathfrak{n}$ and
$T\in F_{0}\mathcal{T}(C(w),E(\sigma))$. Without loss of generality we assume that $Y,Y_{1},\dots,Y_{l}$ are all
$\mathfrak{a}$-weight vectors. For each $i$ ($1\leq i\leq k$), $[Y,Y_{i}]$ is an $\mathfrak{a}$-weight vector
contained in $\mathfrak{n}^{ww_{0}}$. Thus, $[Y,Y_{i}]\in\mathfrak{n}^{ww_{0}}\cap\bar{\mathfrak{n}}$ or
$\mathfrak{n}^{ww_{0}}\cap\mathfrak{n}$. When $[Y,Y_{i}]\in\mathfrak{n}^{ww_{0}}\cap\bar{\mathfrak{n}}$, the
$H$-weight of $Y_{1}\cdots Y_{i-1}[Y,Y_{i}]Y_{i+1}\cdots Y_{l}$ is bigger than $-k_{0}$. Then,
\[(Y_{1}\cdots Y_{i-1}[Y,Y_{i}])\tilde{L}_{Y_{i+1}\cdots Y_{l}}T=\tilde{L}_{Y_{1}\cdots Y_{i-1}[Y,Y_{i}]Y_{i+1}
\cdots Y_{l}}T\in J_{w,\sigma,k_{0}-1}.\] When $[Y,Y_{i}]\in\mathfrak{n}^{ww_{0}}\cap\mathfrak{n}$, let $-k_{1}$
be the $H$-weight of $Y_{1}\cdots Y_{i-1}$ and let $-k_{2}$ be the $H$-weight of $Y_{i+1}\cdots Y_{l}$. Then,
$k_{1}+k_{2}<k_{0}$. By induction we have \[[Y,Y_{i}]\tilde{L}_{Y_{i+1}\cdots Y_{l}}T\in J_{w,\sigma,k_{2}}.\]
Then, \[(Y_{1}\cdots Y_{i-1}[Y,Y_{i}])\tilde{L}_{Y_{i+1}\cdots Y_{l}}T\in J_{w,\sigma,k_{0}-1}.\] For the last
term, we have \[(Y_{1}\cdots Y_{l}Y)T=\tilde{L}_{Y_{1}\cdots Y_{l}}(YT).\] This finishes the proof of \eqref{Eq:YT}.

From \eqref{Eq:YT}, we get \[J_{w,\sigma,k}^{[\mathfrak{n}^{ww_{0}}\cap\mathfrak{n}]}\subset(\mathcal{U}
(\mathfrak{n}^{ww_{0}}\cap\bar{\mathfrak{n}})[k])F_{0}\mathcal{T}(C(w),E(\sigma)^{[\mathfrak{n}^{ww_{0}}\cap
\mathfrak{n}]}+J_{w,\sigma,k-1}.\] By the proof of Lemma \ref{L:Abe1}, we have
\[F_{0}\mathcal{T}(C(w),E(\sigma)^{[\mathfrak{n}^{ww_{0}}\cap\mathfrak{n}]}=F_{0}\mathcal{T}(C(w),
E(\sigma)^{[\mathfrak{n}]}=(R(C(w))\otimes V_{\sigma}^{\ast})\delta_{C(w)}.\] Then,
\[J_{w,\sigma,k}^{[\mathfrak{n}^{ww_{0}}\cap\mathfrak{n}]}\subset J_{w,\sigma}+J_{w,\sigma,k-1}.\] By
\eqref{Eq:Abe4-1} and \eqref{Eq:Abe4-2}, each element of $J_{w,\sigma}$ is annihilated by a power of
$\mathfrak{n}$. Then, \[J_{w,\sigma,k}^{[\mathfrak{n}^{ww_{0}}\cap\mathfrak{n}]}\subset J_{w,\sigma}+
J_{w,\sigma,k-1}^{[\mathfrak{n}^{ww_{0}}\cap\mathfrak{n}]}.\] Proving by induction on $k$, we get
\[\mathcal{T}(C(w),E(\sigma))^{[\mathfrak{n}^{ww_{0}}\cap\mathfrak{n}]}\subset J_{w,\sigma}.\]
Thus, \begin{equation}\label{Eq:Abe4-4}\mathcal{T}(C(w),E(\sigma))^{[\mathfrak{n}]}\subset J_{w,\sigma}.\end{equation}

Combining \eqref{Eq:Abe4-1}, \eqref{Eq:Abe4-2} and \eqref{Eq:Abe4-4}, we get \[\mathcal{T}(C(w),E(\sigma))^{[\mathfrak{n}]}
=J_{w,\sigma}=J'_{w,\sigma}.\]
\end{proof}

For each $p$, set \begin{equation}\label{Jpw}J_{p,w,\sigma}\!=\!\span\{\!\tilde{L}_{Y}\!((f\!\otimes\!u')
\delta_{C(w)}):\!Y\!\in\!\mathcal{U}_{p}(\mathfrak{n}^{ww_{0}}\cap\bar{\mathfrak{n}}),\!f\!\in\!R(C(w)),
\!u'\!\in\!V_{\sigma}^{\ast}\}.\end{equation}

\begin{lemma}\label{L:Abe-Fp}
For each $w\in W$ and each $p\in\mathbb{Z}$, we have \[F_{p}\mathcal{T}(C(w),E(\sigma))^{[\mathfrak{n}]}=J_{p,w,\sigma}.\]
\end{lemma}

\begin{proof}
By Lemma \ref{L:Abe4}, we have \begin{eqnarray*}&&F_{p}\mathcal{T}(C(w),E(\sigma))^{[\mathfrak{n}]}\\&=&
F_{p}\mathcal{T}(C(w),E(\sigma))\cap\mathcal{T}(C(w),E(\sigma))^{[\mathfrak{n}]}\\&=&F_{p}\mathcal{T}(C(w),
E(\sigma))\cap J_{w,\sigma}.\end{eqnarray*} By Lemma \ref{L:Abe-A.3} (ii), \eqref{Jw} and \eqref{Jpw}, the latter
is equal to $J_{p,w,\sigma}$.
\end{proof}

\begin{lemma}\label{L:Abe-Gp}
For any $w\in W$ and each $p\in\mathbb{Z}$, we have \[\Gr^{p}\mathcal{T}(C(w),E(\sigma))^{[\mathfrak{n}]}=
J_{p,w,\sigma}/J_{p-1,w,\sigma}.\]
\end{lemma}

\begin{proof}
There is a short exact sequence \[0\rightarrow F_{p-1}\mathcal{T}(C(w),E(\sigma))\rightarrow F_{p}\mathcal{T}(C(w),
E(\sigma))\rightarrow\Gr^{p}\mathcal{T}(C(w),E(\sigma))\rightarrow 0.\] By the left exactness of the Casselman-Jacquet
functor and Lemma \ref{L:Abe-Fp}, it suffices to show the following assertion: each element in $\Gr^{p}\mathcal{T}(C(w),
E(\sigma))^{[\mathfrak{n}]}$ is the image of some element of $J_{p,w,\sigma}$ in $\Gr^{p}\mathcal{T}(C(w),E(\sigma))$.

By Lemma \ref{L:Abe-A.3}, we have \begin{eqnarray}\label{Eq:Gp-1}&&\Gr^{p}\mathcal{T}(C(w),E(\sigma))\\&=&\nonumber
F_{p}\mathcal{T}(C(w),E(\sigma))/F_{p-1}\mathcal{T}(C(w),E(\sigma))\\&\cong&\nonumber\mathcal{U}_{p}(\mathfrak{n}^{ww_{0}}
\cap\bar{\mathfrak{n}})F_{0}\mathcal{T}(C(w),E(\sigma))/\mathcal{U}_{p-1}(\mathfrak{n}^{ww_{0}}
\cap\bar{\mathfrak{n}})F_{0}\mathcal{T}(C(w),E(\sigma))\\&\cong&\nonumber(\mathcal{U}_{p}(\mathfrak{n}^{ww_{0}}
\cap\bar{\mathfrak{n}})/\mathcal{U}_{p-1}(\mathfrak{n}^{ww_{0}}\cap\bar{\mathfrak{n}}))\otimes_{\mathbb{C}}
F_{0}\mathcal{T}(C(w),E(\sigma)).\end{eqnarray} Following the proof of Lemma \ref{L:Abe4}, we write  \[\mathcal{U}_{p}(\mathfrak{n}^{ww_{0}}\cap\bar{\mathfrak{n}})[k]=\mathcal{U}_{p}(\mathfrak{n}^{ww_{0}}\cap
\bar{\mathfrak{n}})\cap\mathcal{U}(\mathfrak{n}^{ww_{0}}\cap\bar{\mathfrak{n}})[k]\] and put \begin{equation}
\label{Eq:Jw4}J_{p,w,\sigma,k}\!=\!\span\{\tilde{L}_{Y}\!T:\!Y\!\in\mathcal{U}_{p}\!(\mathfrak{n}^{ww_{0}}\!\cap
\bar{\mathfrak{n}})[-k'],\!k'\!\leq\!k,\!T\in F_{0}\!\mathcal{T}\!(C(w),\!E(\sigma))\}.\end{equation} Then,
\[0\subset J_{p,w,\sigma,0}\subset\cdots J_{p,w,\sigma,k}\subset\cdots\] form an exhaustive ascending filtration
of $F_{p}\mathcal{T}(C(w),E(\sigma))$. Moreover, this filtration is compatible with the corresponding filtration of
$F_{p-1}\mathcal{T}(C(w),E(\sigma))$ in the sense that \[J_{p,w,\sigma,k}\cap F_{p-1}\mathcal{T}(C(w),E(\sigma))=
J_{p-1,w,\sigma,k}.\] Then, \[J_{p,w,\sigma,k}+F_{p-1}\mathcal{T}(C(w),E(\sigma))/F_{p-1}\mathcal{T}(C(w),E(\sigma))
\cong J_{p,w,\sigma,k}/J_{p-1,w,\sigma,k}\] form an exhaustive filtration of $\Gr^{p}\mathcal{T}(C(w),E(\sigma))$.
As in the proof of Lemma \ref{L:Abe4}, using \eqref{Eq:Abe4-3} one can show that \begin{eqnarray*}&&((J_{p,w,
\sigma,k}+F_{p-1}\mathcal{T}(C(w),E(\sigma)))/F_{p-1}\mathcal{T}(C(w),E(\sigma)))^{[\mathfrak{n}^{ww_{0}}
\cap\mathfrak{n}]}\\&\subset&(J_{p,w,\sigma}+J_{p,w,\sigma,k-1}+F_{p-1}\mathcal{T}(C(w),E(\sigma)))/
F_{p-1}\mathcal{T}(C(w),E(\sigma))\end{eqnarray*} for each $k\geq 0$. Taking induction on $k$, it follows that
\begin{eqnarray*}&&(F_{p}\mathcal{T}(C(w),E(\sigma))/F_{p-1}\mathcal{T}(C(w),E(\sigma)))^{[\mathfrak{n}^{ww_{0}}
\cap\mathfrak{n}]}\\&\subset&(J_{p,w,\sigma}+F_{p-1}\mathcal{T}(C(w),E(\sigma)))/F_{p-1}\mathcal{T}(C(w),E(\sigma)).
\end{eqnarray*} Thus, any element in $\Gr^{p}\mathcal{T}(C(w),E(\sigma))^{[\mathfrak{n}]}$ is the image of some
element of $J_{p,w,\sigma}$ in $\Gr^{p}\mathcal{T}(C(w),E(\sigma))$.
\end{proof}

\begin{lemma}\label{L:Abe5}
For each $w\in W$ and each $p\in\mathbb{Z}$, there is an exact sequence \begin{eqnarray*}&0&\longrightarrow
F_{p-1}\mathcal{T}(C(w),E(\sigma))^{[\mathfrak{n}]}\longrightarrow F_{p}\mathcal{T}
(C(w),E(\sigma))^{[\mathfrak{n}]}\\&&\longrightarrow\Gr^{p}\mathcal{T}
(C(w),E(\sigma))^{[\mathfrak{n}]}\longrightarrow 0.\end{eqnarray*}
\end{lemma}

\begin{proof}
This follows from Lemmas \ref{L:Abe-Fp} and \ref{L:Abe-Gp} directly.
\end{proof}

\begin{lemma}\label{L:Abe6}
For each $k$, there is an exact sequence \[0\rightarrow I_{k-1,\sigma}^{[\mathfrak{n}]}\rightarrow
I_{k,\sigma}^{[\mathfrak{n}]}\rightarrow(I_{k,\sigma}/I_{k-1,\sigma})^{[\mathfrak{n}]}\rightarrow 0.\]
\end{lemma}

\begin{proof}
We only need to show the surjectivity of the map $I_{k,\sigma}^{[\mathfrak{n}]}\rightarrow (I_{k,\sigma}/I_{k-1,
\sigma})^{[\mathfrak{n}]}$. By Lemma \ref{L:Abe4}, we have \[(I_{k,\sigma}/I_{k-1,\sigma})^{[\mathfrak{n}]}=
\mathcal{T}(C(w_{k}),E(\sigma))^{[\mathfrak{n}]}=J_{w_{k},\sigma}.\] Then, it suffices to find a distribution
in $I_{k,\sigma}^{[\mathfrak{n}]}$ with restriction $\tilde{L}_{Y}((f\otimes u')\delta_{C(w_{k})})$ on $V^{w_{k}}$
for any $Y\in\mathcal{U}(\mathfrak{g})$, $f\in R(C(w_{k}))$ and $u'\in V_{\sigma}^{\ast}$. This is shown in
\cite[Lemma 4.6]{Abe}.
\end{proof}

\begin{lemma}\label{L:Abe7}
Let $0\rightarrow\sigma_{1}\rightarrow\sigma_{2}\rightarrow\sigma_{3}\rightarrow 0$ be a short exact sequence of
finite-dimensional complex linear representations of $P$. Then, the sequence
\[0\rightarrow I(\sigma_{1})^{'[\mathfrak{n}]}\rightarrow I(\sigma_{2})^{'[\mathfrak{n}]}
\rightarrow I(\sigma_{3})^{'[\mathfrak{n}]}\rightarrow 0\] is exact.
\end{lemma}

\begin{proof}
We show that: for each $k\geq 0$, the following sequence \[0\rightarrow I_{k,\sigma_{1}}^{[\mathfrak{n}]}\rightarrow  I_{k,\sigma_{2}}^{[\mathfrak{n}]}\rightarrow I_{k,\sigma_{3}}^{[\mathfrak{n}]}\rightarrow 0\] is exact. By the left
exactness of the Casselman-Jacquet functor, we only need to show that the map $I_{k,\sigma_{2}}^{[\mathfrak{n}]}
\rightarrow I_{k,\sigma_{3}}^{[\mathfrak{n}]}$ is surjective. When $k=0$, we have $I_{k,\sigma_{2}}^{[\mathfrak{n}]}
=I_{k,\sigma_{3}}^{[\mathfrak{n}]}=0$. Hence, the assertion is trivial. Let $l\geq 1$ and suppose the assertion holds
true whenever $k<l$. When $k=l$, we have the following commutative diagram: \begin{equation*}\xymatrix@C=1.0cm{0\ar[r]&I_{l-1,\sigma_{2}}^{[\mathfrak{n}]}\ar[d]\ar[r]&
I_{l,\sigma_{2}}^{[\mathfrak{n}]}\ar[d]\ar[r]&\mathcal{T}(C(w_{l}),E(\sigma_{2}))^{[\mathfrak{n}]}\ar[d]\ar[r]&
0\\0\ar[r]&I_{l-1,\sigma_{3}}^{[\mathfrak{n}]}\ar[r]&I_{l,\sigma_{3}}^{[\mathfrak{n}]}\ar[r]&\mathcal{T}(C(w_{l}),
E(\sigma_{3}))^{[\mathfrak{n}]}\ar[r]&0.}\end{equation*} By Lemma \ref{L:Abe6}, two horizontal lines of this
commutative diagram are short exact sequences. By Lemma \ref{L:Abe4}, the map \[\mathcal{T}(C(w_{l}),
E(\sigma_{2}))^{[\mathfrak{n}]}\rightarrow\mathcal{T}(C(w_{l}),E(\sigma_{3}))^{[\mathfrak{n}]}\] is surjective.
By hypothesis, the map $I_{l-1,\sigma_{2}}^{[\mathfrak{n}]}\rightarrow I_{l-1,\sigma_{3}}^{[\mathfrak{n}]}$ is
surjective. By the five lemma, the surjectivity of the map $I_{l,\sigma_{2}}^{[\mathfrak{n}]}\rightarrow
I_{l,\sigma_{3}}^{[\mathfrak{n}]}$ follows. Taking $k=r$, we get the conclusion of this lemma.
\end{proof}

Among results in this section, Lemmas \ref{L:Abe0}, \ref{L:Abe1}, \ref{L:Abe4} and \ref{L:Abe6} follow from results
in \cite{Abe}. Let's remark on the connection and the difference between our proof and the proof in \cite{Abe}. The
hard part in the proof of Lemma \ref{L:Abe1} is showing that each distribution $(f\otimes u')\delta_{C(w)}$
($f\in R(C(w))$, $u'\in V_{\sigma}^{\ast}$) is annihilated by a power of $\mathfrak{n}$, which is the content of
\cite[\S 3]{Abe}. Different with Abe's proof which defined and used a right action of $\mathcal{U}(\mathfrak{g})$
on $E(\sigma)$-distributions, we only use the left action $\tilde{L}_{\cdot}$ by differential operators. The hard part
in the proof of Lemma \ref{L:Abe4} is showing that \begin{equation*}\mathcal{T}(C(w),E(\sigma))^{[\mathfrak{n}]}
\subset J_{w,\sigma},\end{equation*} which is the content of \cite[\S 2]{Abe}. We avoid the backward induction used
in \cite[Lemma 2.9]{Abe}, and use induction on the $H$-grading degree instead. In another aspect, we consider only
principal series rather than representations induced from a general real parabolic subgroup, and only the
Casselman-Jacquet functor with respect to the nilradical of a minimal real parabolic subalgebra and the trivial
character $\eta=1$. For this reason we avoid complicated computation of differential operator action in \cite{Abe}
and use only elementary representation theory.

\section{A comparison result on nilpotent homogeneous affine variety}\label{S:comparison-poly}

Let $U$ be a closed linear subgroup of $N$. The same as in Lemma \ref{L:Abe1} we have $\mathcal{S}(N/U)'^{[\mathfrak{n}]}
=R(N/U)\delta_{N/U}$. In this section we show that the inclusion $R(N/U)\delta_{N/U}\subset\mathcal{S}(N/U)'$ induces
isomorphisms \[H^{i}(\mathfrak{n},R(N/U)\delta_{N/U})=H^{i}(\mathfrak{n},\mathcal{S}(N/U)'),\ \forall i\geq 0.\] Our
proof is inspired by ideas in \cite[\S 5]{Casselman-Hecht-Milicic}. Note that all closed subgroups of a unipotent
algebraic group are connected.

\begin{lemma}\label{L:reduction1}
Let $U$ be a closed linear subgroup of $N$, and $C$ be a one-dimensional central closed subgroup of $N$ which is not
contained in $U$. Put $V=CU\subset N$. Then \begin{equation}\label{centralcoho2}
H^{i}(\mathfrak{c},\mathcal{S}(N/U)')=\left\{\begin{matrix}\mathcal{S}(N/V)'&\text{for~}i=0\\0&\text{for~}i\neq 0
\end{matrix}\right.\end{equation} and \begin{equation}\label{centralcoho}
H^{i}(\mathfrak{c},R(N/U)\delta_{N/U})=\left\{\begin{matrix}R(N/V)\delta_{N/V}&\text{for~}i=0\\0
&\text{for~}i\neq 0.\end{matrix}\right.\end{equation}
\end{lemma}

\begin{proof}
Let $p:N/U\longrightarrow N/V$ be the natural projection. There is a natural short exact sequence: \begin{equation*}
\xymatrix@R=0.05cm{0\ar[r]&\mathcal{S}(N/U)\ar[r]^{\xi}&\mathcal{S}(N/U)\ar[r]^{\pi}&\mathcal{S}(N/V)\ar[r]&0,}
\end{equation*} where $\xi$ is an $N$ invariant differential operator on $N/U$ associated to a nonzero element
$\xi\in\mathfrak{z}_{0}$ and the map $\pi$ is defined by \[\pi(h)(p(x))=\int_{C}h(xz)\d z\ (\forall h\in
\mathcal{S}(N/U)), \forall x\in N/U.\] Then, it induces the following exact sequences: \begin{equation*}
\xymatrix@R=0.05cm{0\ar[r]&\mathcal{S}(N/V)'\ar[r]^{\pi^*}&\mathcal{S}(N/U)'\ar[r]^{\xi^*}&\mathcal{S}(N/U)'
\ar[r]&0}\end{equation*} and \begin{equation*}
\xymatrix@R=0.05cm{0\ar[r]& R(N/V)\delta_{N/V}\ar[r]^{\pi^*}& R(N/U)\delta_{N/U}\ar[r]^{\xi^*}& R(N/U)\delta_{N/U}
\ar[r]&0}.\end{equation*} Note that for each $f\in R(N/V)$ and each $h\in\mathcal{S}(N/U)$, we have
\begin{eqnarray*}
&&\quad\langle\pi^{*}(f\delta_{N/V}),h\rangle\\
&&=\langle f\delta_{N/V},\pi(h)\rangle\\
&&=\int_{N/V}f(y)\pi(h)(y)\d y\\
&&=\int_{N/V}f(y)\int_{V/U}h(yz)\d z\d y\\
&&=\int_{N/U}\tilde{f}(x)h(x)\d x\\
&&=\langle\tilde{f}\delta_{N/U},h\rangle
\end{eqnarray*}
where $\tilde{f}\in R(N/U)$ is given by $\tilde{f}(x)=f(p(x))\ (\forall x\in N/U)$. Then, $\pi^*(f\delta_{N/V})
=\tilde{f}\delta_{N/U}$. Consequently, \begin{equation*}
H^{i}(\mathfrak{c},\mathcal{S}(N/U)')=\left\{\begin{matrix}\mathcal{S}(N/V)'&\text{for~}i=0\\0&\text{for~}i\neq 0
\end{matrix}\right.\end{equation*}
and \begin{equation*}
H^{i}(\mathfrak{c},R(N/U)\delta_{N/U})=\left\{\begin{matrix}R(N/V)\delta_{N/V}&\text{for~}i=0\\0
&\text{for~}i\neq 0.\end{matrix}\right.\end{equation*}
\end{proof}

\begin{lemma}\label{L:reduction2}
Let $D$ be an abelian real linear group. Then we have \begin{equation}\label{D:S2}
H^{j}(\mathfrak{d},\mathcal{S}(D)')=\left\{\begin{matrix}\mathbb{C}\delta_{D} &\text{for~}j=0\\0&\text{for~}j\neq 0
\end{matrix}\right.\end{equation} and \begin{equation}\label{D:R2}
H^{j}(\mathfrak{d},R(D)\delta_{D})=\left\{\begin{matrix}\mathbb{C}\delta_{D} &\text{for~}j=0\\0&\text{for~}j\neq 0.
\end{matrix}\right.\end{equation}
\end{lemma}

\begin{proof}
We show \eqref{D:R2}. The equality \eqref{D:S2} can be proved in the same way. Prove by induction on $\dim D$. When
$\dim D=0$, \eqref{D:R2} is trivial. Suppose \eqref{D:R2} holds whenever $\dim D<k$. Now let $\dim D=k\geq 1$.
Choose a one-dimensional closed linear subgroup $Z$ of $D$. By \eqref{centralcoho} the Hochschild-Serre spectral
sequence \[H^{p}(\mathfrak{d}/\mathfrak{z},H^{q}(\mathfrak{z},R(D)\delta_{D}))\Longrightarrow H^{p+q}(\mathfrak{d},
R(D)\delta_{D})\] degenerates and \[H^{p}(\mathfrak{d},R(D)\delta_{D})=H^{p}(\mathfrak{d}/\mathfrak{z},R(D/Z)
\delta_{D/Z}).\] Then, the conclusion follows from the induction hypothesis.
\end{proof}

\begin{proposition}\label{P:comparison1}
Let $U$ be a closed linear subgroup $N$. Then the inclusion $R(N/U)\delta_{N/U}\subset\mathcal{S}(N/U)'$ induces
isomorphisms \[H^{i}(\mathfrak{n},R(N/U)\delta_{N/U})=H^{i}(\mathfrak{n},\mathcal{S}(N/U)'),\ \forall i\geq 0.\]
\end{proposition}

\begin{proof} We prove the proposition by induction on $\dim N$.

\textbf{Case~I:~$N$ is abelian.} In this case, $N=U\times D$ for some complementary abelian Lie subgroup
$D$. The subgroup $U$ acts on $\mathcal{S}(N/U)'$ and $R(N/U)\delta_{N/U}$ trivially. Then, the differentials in the
Koszul complexes that compute $H^{\ast}(\mathfrak{u},\mathcal{S}(N/U)')$ and $H^{\ast}(\mathfrak{u},\mathcal{S}(N/U)')$
are all 0. Thus, \begin{equation}\label{D:S1}
H^{j}(\mathfrak{u},\mathcal{S}(N/U)')=\bigwedge^{j}\mathfrak{u}^*\otimes\mathcal{S}(N/U)',~~~j\in\mathbb{Z}.
\end{equation} and \begin{equation}\label{D:R1}
H^{j}(\mathfrak{u},R(N/U)\delta_{N/U})=\bigwedge^{j}\mathfrak{u}^*\otimes R(N/U)\delta_{N/U},~~~j\in\mathbb{Z}.
\end{equation}
By \eqref{D:S2} and \eqref{D:R2}, $H^{p}(\mathfrak{n}/\mathfrak{u},H^{q}(\mathfrak{u},\cdot))\neq 0$ only when
$p=0$ while $\cdot=\mathcal{S}(N/U)'$ or $R(N/U)\delta_{N/U}$. Then, both Hochschild-Serre spectral sequences \[H^{p}(\mathfrak{n}/\mathfrak{u},H^{q}(\mathfrak{u},\cdot))\Rightarrow H^{p+q}(\mathfrak{n},\cdot)\] for
$\cdot=\mathcal{S}(N/U)'$ and $R(N/U)\delta_{N/U}$ degenerate. Thus, we have \[H^{i}(\mathfrak{n},\mathcal{S}(N/U)')
=H^{0}(\mathfrak{n}/\mathfrak{u},\bigwedge^{i}\mathfrak{u}^*\otimes\mathcal{S}(N/U)')=\bigwedge^{i}\mathfrak{u}^*,
\ \forall i\geq 0\] and \[H^{i}(\mathfrak{n},R(N/U)\delta_{N/U})=H^{0}(\mathfrak{n}/\mathfrak{u},
\bigwedge^{i}\mathfrak{u}^{*}\otimes R(N/U)\delta_{N/U}')=\bigwedge^{i}\mathfrak{u}^{*},\ \forall i\geq 0.\] Hence,
we get isomorphisms $H^{i}(\mathfrak{n},R(N/U)\delta_{N/U})=H^{i}(\mathfrak{n},\mathcal{S}(N/U)')$ ($i\geq 0$).

\textbf{Case II:~$N$ is non-abelian.} Choose a one-dimensional closed linear central subgroup $C$ of $N$. First,
assume that $C\subset U$. Put $N'=N/C$ and $U'=U/C$. Then, $N/U=N'/U'$. Note that the action of $C$ on
$R(N/U)\delta_{N/U}$ is trivial. Hence, \begin{equation}\label{D:S3}
H^{i}(\mathfrak{c},\mathcal{S}(N/U)')=\left\{\begin{matrix}\mathcal{S}(N/U)'&i=0,1\\0&i\neq 0,1.\end{matrix}\right.
\end{equation} and \begin{equation}\label{D:R3}
H^{i}(\mathfrak{c},R(N/U)\delta_{N/U})=\left\{\begin{matrix} R(N/U)\delta_{N/U}&\text{for~}i=0,1
\\0&\text{for~}i\neq 0,1. \end{matrix}\right.\end{equation} Then, the Hochschild-Serre spectral sequences \[H^{p}(\mathfrak{n}',H^{q}(\mathfrak{c},R(N/U)\delta_{N/U}))\Longrightarrow H^{p+q}(\mathfrak{n},R(N/U)\delta_{N/U})\]
and \[H^{p}(\mathfrak{n}',H^{q}(\mathfrak{c},\mathcal{S}(N/U)'))\Longrightarrow H^{p+q}(\mathfrak{n},\mathcal{S}(N/U)')\]
degenerate into short exact sequences and we are led to the following commutative diagram:
\begin{equation*}
\xymatrix@C=0.3cm{0\ar[r]&H^{i}(\mathfrak{n}',R(N'/U')\delta_{N'/U'})
\ar[d]\ar[r]&H^{i}(\mathfrak{n},R(N/U)\delta_{N/U})\ar[d]&\\
0\ar[r]&H^{i}(\mathfrak{n}',\mathcal{S}(N'/U')')\ar[r]&H^{i}(\mathfrak{n},\mathcal{S}(N/U)')&\\
&\ar[r]&H^{i-1}(\mathfrak{n}',R(N'/U')\delta_{N'/U'})\ar[d]\ar[r]&0\\
&\ar[r]&H^{i-1}(\mathfrak{n}',\mathcal{S}(N'/U')))\ar[r]&0}.
\end{equation*}
By induction hypothesis, the first and the third vertical arrows in the above commutative diagram are isomorphisms.
By the five lemma, we get isomorphisms $H^{i}(\mathfrak{n},R(N/U)\delta_{N/U})=H^{i}(\mathfrak{n},\mathcal{S}(N/U)')$
($\forall i\geq 0$).

Second, assume that $C\nsubseteq U$. Put $N'=N/C$ and $U'=UC/C$. Let $V=UC$. By \eqref{centralcoho2} and
\eqref{centralcoho}, the Hochschild-Serre spectral sequences
\[H^{p}(\mathfrak{n}/\mathfrak{c},H^{q}(\mathfrak{c},\ast))\Rightarrow H^{p+q}(\mathfrak{n},\ast)\] for $\ast=
\mathcal{S}(N/U)'$ and $R(N/U)\delta_{N/U}$ degenerate. Then, we have \[H^{i}(\mathfrak{n},\mathcal{S}(N/U)')
=H^{i}(\mathfrak{n}',\mathcal{S}(N'/U')'),\ \forall i\geq 0\] and \[H^{i}(\mathfrak{n},R(N/U)\delta_{N/U})=
H^{i}(\mathfrak{n}',R(N'/U')\delta_{N'/U'}),\ \forall i\geq 0.\] Since $\dim N'=\dim N-1<\dim N$, by induction
hypothesis we get isomorphisms $H^{i}(\mathfrak{n},R(N/U)\delta_{N/U})=H^{i}(\mathfrak{n},\mathcal{S}(N/U)')$
($\forall i\geq 0$).
\end{proof}

\section{Casselman's comparison theorem}

\subsection{Comparison theorems}

\begin{theorem}\label{T:comparison3}
Let $(\sigma,V_{\sigma})$ be a finite-dimensional complex linear algebraic representation of $P$. Then the
inclusion $I(\sigma)'\rightarrow I(\sigma)_{K}^{\ast}$ induces isomorphisms \[H^{i}(\mathfrak{n},I(\sigma)')=
H^{i}(\mathfrak{n},I(\sigma)_{K}^{\ast}),\ \forall i\geq 0.\]
\end{theorem}

\begin{proof}
By \cite[Lemma 2.37]{Hecht-Schmid}, we have \begin{equation}\label{Eq:HS}H^{i}(\mathfrak{n},I(\sigma)_{K}^{\ast})
=H^{i}(\mathfrak{n},(I(\sigma)_{K}^{\ast})^{[\mathfrak{n}]}),\ \forall i\geq 0.\end{equation} By Casselman's
automatic continuity theorem (\cite[p. 416]{Casselman},\cite[Theorem 11.4]{Bernstein-Krotz}, \cite[p.77]{Wallach}),
we have \begin{equation}\label{Eq:continuity}I(\sigma)^{'[\mathfrak{n}]}=(I(\sigma)_{K}^{\ast})^{[\mathfrak{n}]}.
\end{equation} Then, it reduces to show the following assertion: the inclusion $I(\sigma)^{'[\mathfrak{n}]}
\subset I(\sigma)$ induces isomorphisms \begin{equation}\label{Eq:comparison3}H^{i}(\mathfrak{n},
I(\sigma)^{'[\mathfrak{n}]})=H^{i}(\mathfrak{n},I(\sigma)'),\ \forall i\geq 0.\end{equation}

First, suppose that the action of $N$ on $V_{\sigma}$ is trivial. Since $I(\sigma)=C^{\infty}(E(\sigma))$, then
$I(\sigma)'=C^{\infty}(E(\sigma))'$ and $I(\sigma)^{'[\mathfrak{n}]}=C^{\infty}(E(\sigma))^{'[\mathfrak{n}]}$.
Take the Bruhat filtration of $C^{\infty}(E(\sigma))'$. By Lemma \ref{L:Abe6}, it suffices to prove that: for
each $k\in\mathbb{Z}$, the inclusion $(I_{k}/I_{k-1})^{[\mathfrak{n}]}\subset I_{k}/I_{k-1}$ induces isomorphisms
for cohomology on all degrees. Since $I_{k}/I_{k-1}=\mathcal{T}(C(w_{k}),E(\sigma))$, it suffices to prove the
cohomological isomorphism for the inclusions $\mathcal{T}(C(w),E(\sigma))^{[\mathfrak{n}]}\subset\mathcal{T}(C(w),
E(\sigma))$ ($w\in W$). Consider the filtration through transversal degree. Since Lie algebra cohomology is
compatible with direct limit, it suffices to show the cohomological isomorphism for each inclusion
$(F_{p}\mathcal{T}(C(w),E))^{[\mathfrak{n}]}\subset F_{p}\mathcal{T}(C(w),E)$ ($p\in\mathbb{Z}_{\geq 0}$).
By Lemma \ref{L:Abe5} and taking induction on the degree $p$, it suffices to show that the inclusions
$(\Gr^{p}\mathcal{T}(C(w),E))^{[\mathfrak{n}]}\subset\Gr^{p}\mathcal{T}(C(w),E)$ induces isomorphism for
cohomology on all degrees. By Corollary \ref{C:free}, it suffices to show so for the inclusions $\mathcal{S}(C(w),
E(\sigma))^{'[\mathfrak{n}]}\subset\mathcal{S}(C(w),E(\sigma))'$ ($w\in W$), which follows from Lemma \ref{L:Abe1}
and Proposition \ref{P:comparison1}.

In general, take a composition series of $V_{\sigma}$ as a $P$ representation: \[0=V_{0}\subset V_{1}\subset\cdots
\subset V_{k}=V_{\sigma}.\] Then, the action of $N$ on each graded piece $V_{j}/V_{j-1}$ ($1\leq j\leq k$) is
trivial. Thus, we have isomorphisms \[H^{i}(\mathfrak{n},I(V_{j}/V_{j-1})^{'[\mathfrak{n}]})=H^{i}(\mathfrak{n},
I(V_{j}/V_{j-1})'),\ \forall i\in\mathbb{Z}.\] By Lemma \ref{L:Abe7}, we have short exact sequences \[0\rightarrow
I(V_{j-1})^{'[\mathfrak{n}]}\rightarrow I(V_{j})^{'[\mathfrak{n}]}\rightarrow I(V_{j}/V_{j-1})^{'[\mathfrak{n}]}
\rightarrow 0.\] By the cohomological long exact sequences associated to short exact sequences \[0\rightarrow
I(V_{j-1})'\rightarrow I(V_{j})'\rightarrow I(V_{j}/V_{j-1})'\rightarrow 0\] and \[0\rightarrow
I(V_{j-1})^{'[\mathfrak{n}]}\rightarrow I(V_{j})^{'[\mathfrak{n}]}\rightarrow I(V_{j}/V_{j-1})^{'[\mathfrak{n}]}
\rightarrow 0\] and taking induction on $j$, one shows that each inclusion $I(V_{j})^{'[\mathfrak{n}]}\rightarrow
I(V_{j})'$ induces isomorphisms \[H^{i}(\mathfrak{n},I(V_{j})^{'[\mathfrak{n}]})=H^{i}(\mathfrak{n},I(V_{j})'),
\ \forall i\in\mathbb{Z}.\] Taking $j=k$, we get the conclusion of the theorem.
\end{proof}

\begin{theorem}\label{T:comparison4}
Let $V$ be an admissible finitely generated moderate growth smooth Fr\'echet representation of $G$. Then the inclusion
$V_{K}\subset V$ induces isomorphisms \begin{equation}\label{Eq:homology}H_{i}(\mathfrak{n},V_{K})=H_{i}(\mathfrak{n},
V),\ \forall i\in\mathbb{Z}\end{equation} and the inclusion $V'\rightarrow V_{K}^{\ast}$ induces isomorphisms
\begin{equation}\label{Eq:cohomology}H^{i}(\mathfrak{n},V')=H^{i}(\mathfrak{n},V_{K}^{\ast}),\ \forall i\in\mathbb{Z}.
\end{equation}
\end{theorem}

\begin{proof}
First, we show the homological comparison theorem \eqref{Eq:homology} and the cohomological comparison theorem
\eqref{Eq:cohomology} are equivalent. Suppose we have the cohomological comparison theorem. Then, $H^{i}(\mathfrak{n},V)$
are all finite-dimensional as $H^{i}(\mathfrak{n},V_{K})$ are (\cite[Corollary 2.4]{Casselman-Osborne}). By a duality
argument in \cite[Lemma 5.11]{Casselman-Hecht-Milicic}, we get the homological comparison theorem. The proof for the
converse direction is similar. We show the homological comparison theorem \eqref{Eq:homology} below.

Second, when $V=I(\sigma)$ is a principal series with $\sigma$ a finite-dimensional complex linear algebraic
representation of $P$, \eqref{Eq:cohomology} is shown in Theorem \ref{T:comparison3}. By the duality argument
in \cite[Lemma 5.11]{Casselman-Hecht-Milicic}, we get \eqref{Eq:homology} for $V=I(\sigma)$.

Third, by \cite[Proposition 3]{Hecht-Taylor3}, one can reduce the homological comparison theorem for a general
representation $V\in\mathcal{HF}_{\mod}(G)$ to the principal series case. For readers' convenience, we recall
Hecht-Taylor's proof in \cite{Hecht-Taylor3}. By Casselman's subrepresentation theorem there is a finite-dimensional 
complex linear algebraic representation $\sigma_{0}$ of $P$ such that there is an injection $V\hookrightarrow 
I(\sigma_{0})$ (\cite[Proposition 4.2.3]{Wallach1}). Let $V_{1}$ be the cokernel, which is still in the category
$\mathcal{HF}_{\mod}(G)$. By Casselman's subrepresentation theorem again, there is a finite-dimensional complex
linear algebraic representation $\sigma_{1}$ of $P$ such that there is an injection $V_{1}\hookrightarrow
I(\sigma_{1})$. Continuing in this way, we obtain a resolution of $V$ by principal series:\[V\rightarrow 
I(\sigma_{0})\rightarrow I(\sigma_{1})\rightarrow\cdots.\] This double complex has bounded $p$ parameter. 
The natural map of double complexes \[\wedge^{p}\mathfrak{n}\otimes I(\sigma_{q})_{K}\rightarrow\wedge^{p}
\mathfrak{n}\otimes I(\sigma_{q})\] can be analyzed by means of the (convergent) spectral sequences associated 
to the two standard filtrations. The spectral sequence corresponding to the second filtration degenerates at 
$E_{2}: E_{2}^{q,-p}$-terms are zero except when $q=0$, and the resulting map on the $(-p)$-cohomology of the 
total complex is nothing but \begin{equation}\label{Eq:filtration2'}H_{p}(\mathfrak{n},V_{K})\rightarrow 
H_{p}(\mathfrak{n},V).\end{equation} On the other hand, the map between the $E_{1}^{-p,q}$ terms of the 
spectral sequence associated to the first filtration is \begin{equation}\label{Eq:filtration1'}
H_{p}(\mathfrak{n},I(\sigma_{q})_{K})\rightarrow H_{p}(\mathfrak{n},I(\sigma_{q})).\end{equation} As shown 
above, maps in \eqref{Eq:filtration1'} are isomorphisms, so are maps in \eqref{Eq:filtration2'}.
\end{proof}

\subsection{Closedness}

Theorem \ref{T:comparison4} has the following immediate consequence.

\begin{corollary}\label{C:comparison5}
Let $V$ be an admissible finitely generated  moderate growth smooth Fr\'echet representation of $G$. For the Koszul
complex associated to $V$: \begin{equation*}\xymatrix@R=0.05cm{0\ar[r]&\wedge^{d}\mathfrak{n}\otimes_{\mathbb{C}}V
\ar[r]^{\partial_{d}}&\cdots\ar[r]^{\partial_{2}}&\mathfrak{n}\otimes_{\mathbb{C}}V\ar[r]^{\partial_{1}}&V\ar[r]&0}
\end{equation*} where $d=\dim\mathfrak{n}$, each $\Im(\partial_{i})$ is an NF space and it is a closed subspace of
$\wedge^{i-1}\mathfrak{n}\otimes_{\mathbb{C}}V$ ($1\leq i\leq d$). Similarly, we have DNF property and closedness
for images of boundary operators in the Koszul complex defining $\mathfrak{n}$-cohomology of $V'$.
\end{corollary}

\begin{proof}
By Theorem \ref{T:comparison4}, $\Im(\partial_{i})$ is a finite co-dimensional subspace of $\ker(\partial_{i-1})$.
Since $\partial_{i-1}$ is a linear continuous map, then $\ker(\partial_{i-1})$ is a closed subspace of the NF space $\wedge^{i-1}\mathfrak{n}\otimes_{\mathbb{C}}V$. Thus, $\ker(\partial_{i-1})$ is also an NF space. By Theorem
\ref{T:comparison3}, the linear continuous map $\partial_{i}:\wedge^{i}\mathfrak{n}\otimes_{\mathbb{C}}V\rightarrow
\ker(\partial_{i-1})$ has finite co-dimensional image for each $i$. By \cite[Lemma A.1]{Casselman-Hecht-Milicic},
it follows that $\Im(\partial_{i})$ is an NF space. Hence, it is a closed subspace of $\wedge^{i-1}\mathfrak{n}
\otimes_{\mathbb{C}}V$. The proof for the DNF property and closedness of images of boundary operators in the Koszul
complex defining $\mathfrak{n}$-cohomology of $V'$ is similar.
\end{proof}

Let $V$ be an admissible finitely generated  moderate growth smooth Fr\'echet representation of $G$. Recall that
Casselman's automatic continuity theorem says that \[V_{K}/\mathfrak{n}^{k}V_{K}=V/\overline{\mathfrak{n}^{k}V},
\ \forall k\geq 0.\] In the following Theorem \ref{T:comparison4}, we show that each $\mathfrak{n}^{k}V$ is a
closed subspace of $V$. Hence, $V_{K}/\mathfrak{n}^{k}V_{K}=V/\mathfrak{n}^{k} V$.

\begin{theorem}\label{T:comparison6}
Let $V$ be an admissible finitely generated  moderate growth smooth Fr\'echet representation of $G$. Then for every
$k\geq 0$, $\mathfrak{n}^{k} V$ is a closed subspace of $V$ and the inclusion $V_{K}\subset V$ induces an isomorphism
\[V_{K}/\mathfrak{n}^{k}V_{K}=V/\mathfrak{n}^{k} V.\]
\end{theorem}

\begin{proof}
By the homological comparison theorem in Theorem \ref{T:comparison4}, $\mathfrak{n}V$ is a finite co-dimensional
subspace of $V$. Thus, there is a finite-dimension subspace $U$ of $V$ such that $V=\mathfrak{n}V+U$. Hence, for
each $k\geq 0$, \[V=\mathfrak{n}^{k}V+(\sum_{0\leq j\leq k-1}\mathfrak{n}^{j}U).\] By this, $\mathfrak{n}^{k}V$ is
a finite co-dimensional subspace of $V$. Then, $\mathfrak{n}^{k}\otimes V\rightarrow V$ is a linear continuous map
having a finite co-dimensional image. By \cite[Lemma A.1]{Casselman-Hecht-Milicic}, it follows that $\mathfrak{n}^{k}V$
is a closed subspace of $V$. By Casselman's automatic continuity theorem we have $V_{K}/\mathfrak{n}^{k}V_{K}=
V/\overline{\mathfrak{n}^{k}V}$. Then, it follows that $V_{K}/\mathfrak{n}^{k}V_{K}=V/\mathfrak{n}^{k} V$.
\end{proof}


\begin{thebibliography}{99}
\bibitem{Abe} N.~Abe, \emph{Generalized Jacquet modules of parabolically induced representations.} Publ. Res. Inst. Math. Sci.
\textbf{48} (2012), no. 2, 419-473.
\bibitem{Bernstein-Krotz} J.~Bernstein; B.~Kr\"otz, \emph{Smooth Fr\'echet globalizations of Harish-Chandra modules.} Israel
J. Math. \textbf{199} (2014), no. 1, 45-111.
\bibitem{Bierstone-Milman} E.~Bierstone; P.D.~Milman, \emph{Semianalytic and subanalytic sets.} Inst. Hautes \'Etudes Sci.
Publ. Math. (1988), no. 67, 5-42.
\bibitem{Bratten}T.~Bratten, \emph{A comparison theorem for Lie algebra homology groups.} Pacific J. Math. \textbf{182} (1998),
no. 1, 23-36.
\bibitem{Bunke-Olbrich} U.~Bunke; M.~Olbrich, \emph{Cohomological properties of the canonical globalizations of
Harish-Chandra modules. Consequences of theorems of Kashiwara-Schmid, Casselman, and Schneider-Stuhler.} (English summary)
Ann. Global Anal. Geom. \textbf{15} (1997), no. 5, 401-418.
\bibitem {Casselman2} W.~Casselman, \emph{Jacquet modules for real reductive groups.} Proceedings of the International
Congress of Mathematicians (Helsinki, 1978), Acad. Sci. Fennica, Helsinki, 1980, 557-563.
\bibitem {Casselman} W.~Casselman, \emph{Canonical extensions of Harish-Chandra modules to representations of $G$.} Canad.
J. Math. \textbf{41} (1989), no. 3, 385-438.
\bibitem {Casselman-Hecht-Milicic} W.~Casselman; H.~Hecht; D.~Mili$\check{c}$i\'c, \emph{Bruhat filtrations and Whittaker
vectors for real groups.} The mathematical legacy of Harish-Chandra (Baltimore, MD, 1998), 151-190, Proc. Sympos. Pure Math.,
\textbf{68}, Amer. Math. Soc., Providence, RI, 2000.
\bibitem{Casselman-Osborne}W.~Casselman; M.S.~Osborne, \emph{The restriction of admissible representations to $\mathfrak{n}$.}
Math. Ann. \textbf{233} (1978), no. 3, 193-198.
\bibitem {Hecht-Schmid} H.~Hecht; W.~Schmid, \emph{Characters, asymptotics and $\mathfrak{n}$-homology of Harish-Chandra modules.}
Acta Math. \textbf{151} (1983), no. 1-2, 49-151.
\bibitem{Hecht-Taylor} H.~Hecht; J.~Taylor, \emph{A comparison theorem for $\mathfrak{n}$-homology.} Compositio Math. \textbf{86}
(1993), no. 2, 189-207.
\bibitem{Hecht-Taylor3} H.~Hecht; J.~Taylor, \emph{A remark on Casselman's comparison theorem.} Geometry and Representation
Theory of Real and $p$-adic Groups, C\'ordoba, 1995, Progr. Math., vol. \textbf{158}, Birkh\"auser Boston, Boston, MA, 1998,
pp. 139-146.
\bibitem{Kashiwara-Schmid} M.~Kashiwara; W.~Schmid, \emph{Quasi-equivariant D-modules, equivariant derived category, and representations
of reductive Lie groups.} Lie theory and geometry, 457-488, Progr. Math., \textbf{123}, Birkh\"{a}user Boston, Boston, MA, 1994.
\bibitem{Knapp}A.W.~Knapp, \emph{Representation theory of semisimple groups. An overview based on examples.} Reprint of the 1986
original. Princeton Landmarks in Mathematics. Princeton University Press, Princeton, NJ, 2001.
\bibitem{Schaefer} H.~Schaefer, \emph{Topological vector spaces.} Macmillan, 1966.
\bibitem{Schmid} W.~Schmid, \emph{Boundary value problems for group invariant differential equations.} Proc. Cartan Symposium,
Ast\'erique, 1985.
\bibitem{Schmid-Wolf} W.~Schmid; J.~Wolf, \emph{Geometric quantization and derived functor modules for semisimple Lie groups.}
J. Funct. Anal. \textbf{90} (1990), no. 1, 48-112.
\bibitem{Schwartz} L.~Schwartz, \emph{Th\'{e}orie des distributions.} Hermann, Paris, 1966.
\bibitem{Vogan} D.~Vogan, \emph{Unitary representations and complex analysis.} Representation theory and complex analysis, 259-344,
Lecture Notes in Math., \textbf{1931}, Springer, Berlin, 2008.
\bibitem{Wallach1} N.~Wallach, \emph{Real reductive groups. I.} Pure and Applied Mathematics, \textbf{132}. Academic Press, Inc.,
Boston, MA, 1988.
\bibitem{Wallach} N.~Wallach, \emph{Real reductive groups. II.} Pure and Applied Mathematics, \textbf{132-II}. Academic Press,
Inc., Boston, MA, 1992.
\bibitem{Wallach2} N.~Wallach, \emph{Fr\'echet completions of moderate growth old and (somewhat) new results.} arXiv:1403.3658.
\end{thebibliography}
\end{document}